\documentclass[12pt]{amsart}

\calclayout

\usepackage{amsmath,amsthm,amssymb}
\usepackage{mathrsfs}

\usepackage{colonequals}
\usepackage{enumitem}

\usepackage[unicode,colorlinks=true,allcolors=blue,pdftitle={On the first critical in 3D Ginzburg-Landau}]{hyperref}

\usepackage[initials,nobysame,alphabetic,msc-links,backrefs]{amsrefs}

\newtheorem{theorem}{Theorem}[section]

\newtheorem{proposition}{Proposition}[section]
\newtheorem{lemma}{Lemma}[section]
\newtheorem{definition}{Definition}[section]
\newtheorem{remark}{Remark}[section]

\numberwithin{equation}{section}

\DeclareMathOperator{\diver}{div}
\DeclareMathOperator{\curl}{curl}
\DeclareMathOperator{\curlcurl}{curl^2}

\newcommand{\de}{\delta}

\newcommand{\ve}{\varepsilon}
\newcommand{\vart}{\vartheta}

\newcommand{\BB}{\mathcal B}

\newcommand{\N}{\mathbb N}
\newcommand{\R}{\mathbb R}
\renewcommand{\C}{\mathbb C}

\newcommand{\ex}{\mathrm{ex}}
\newcommand{\loc}{\mathrm{loc}}
\newcommand{\D}{\mathrm D}

\newcommand{\GG}{\mathfrak G}
\newcommand{\RR}{\mathfrak R}

\renewcommand{\H}{\mathcal H}

\title[On the first critical field in 3D Ginzburg-Landau]{On the first critical field in the 3D Ginzburg-Landau model of superconductivity}
\subjclass[2010]{35J20,35J25,35J50,35J60,35Q56,49Q15,49Q20,53Z05,82D55}
\keywords{Ginzburg-Landau, first critical field, Meissner solution, energy minimizers}
\author{Carlos Rom\'{a}n}
\date{February 26, 2018}
\address{Mathematisches Institut\\
Universit\"{a}t Leipzig\\
Augustusplatz 10, 04109 Leipzig, Germany}
\urladdr{\href{http://www.math.uni-leipzig.de/~roman/}{www.math.uni-leipzig.de/~roman/}}
\email{\href{mailto:roman@math.uni-leipzig.de}{roman@math.uni-leipzig.de}}

\begin{document}
\begin{abstract}
The Ginzburg-Landau model is a phenomenological description of superconductivity. A crucial feature of type-II superconductors is the occurrence of vortices, which appear above a certain value of the applied magnetic field called the first critical field. In this paper we estimate this value, when the Ginzburg-Landau parameter is large, and we characterize the behavior of the Meissner solution, the unique vortexless configuration that globally minimizes the Ginzburg-Landau energy below the first critical field. In addition, we show that beyond this value, for a certain range of the strength of the applied field, there exists a unique Meissner-type solution that locally minimizes the energy.
\end{abstract}
\maketitle 

%%%%%%%%%%%%%%%%%%%%%%%%%%%%%%%%%%%%%%%%%%%%%%%%%%%%%%%%%%%%%%%%%%%%%%%%%%%%%%%%%%%%%%%%%%%%%%%%%

\section{Introduction}
\subsection{Problem and background}
Superconductors are certain metals and alloys, which, when cooled down below a critical (typically very low) temperature, lose their resistivity, which allows permanent currents to circulate without loss of energy. Superconductivity was discovered by Ohnes in 1911. As a phenomenological description of this phenomenon, Ginzburg and Landau \cite{GinLan} introduced in 1950 the Ginzburg-Landau model of superconductivity, which has been proven to effectively predict the behavior of superconductors and that was subsequently justified as a limit of the Bardeen-Cooper-Schrieffer (BCS) quantum theory \cite{BCS}. It is a model of great importance in physics, with Nobel prizes awarded for it to Abrikosov, Ginzburg, and Landau.

The Ginzburg-Landau functional, which models the state of a superconducting sample in an applied magnetic field, assuming that the temperature is fixed and below the critical one, is
$$
GL_\ve(u,A)=\frac12\int_\Omega |\nabla_A u|^2+\frac{1}{2\ve^2}(1-|u|^2)^2+\frac12\int_{\R^3}|H-H_{\ex}|^2.
$$
Here
\begin{itemize}
\item $\Omega$ is a bounded domain of $\R^3$, that we assume to be simply connected with $C^2$ boundary.
\item $u:\Omega\rightarrow \mathbb{C}$ is called the \emph{order parameter}. Its modulus squared (the density of Cooper pairs of superconducting electrons in the BCS quantum theory) indicates the local state of the superconductor: where $|u|^2\approx 1$ the material is in the superconducting phase, where $|u|^2\approx 0$ in the normal phase.
\item $A:\R^3\rightarrow \R^3$ is the electromagnetic vector potential of the induced magnetic field $H=\curl A$.
\item $\nabla_A$ denotes the covariant gradient $\nabla-iA$.
\item $H_{\ex}:\R^3\rightarrow \R^3$ is a given external (or applied) magnetic field.
\item $\ve>0$ is the inverse of the \emph{Ginzburg-Landau parameter} usually denoted $\kappa$, a non-dimensional parameter depending only on the material. We will be interested in the regime of small $\ve$, corresponding to extreme type-II superconductors. 
\end{itemize}

A key physical feature of this type of superconductors is the occurrence of co-dimension 2 topological singularities called \emph{vortices}, which appear above a certain critical value of the strength of the applied field $h_{\ex}\colonequals \|H_{\ex}\|_{L^2(\Omega,\R^3)}$. There are three main critical values of $h_{\ex}$ or critical fields $H_{c_1},H_{c_2}$, and $H_{c_3}$, for which phase transitions occur.
\begin{itemize}
\item Below $H_{c_1}=O(|\log \ve|)$, the superconductor is everywhere in its superconducting phase, i.e. $|u|$ is uniformly close to $1$, and the applied field is expelled by the material due to the occurrence of supercurrents near $\partial\Omega$. This phenomenon is known as the \emph{Meissner effect}.
\item At $H_{c_1}$, the first vortice(s) appear and the applied field penetrates the superconductor through the vortice(s).
\item Between $H_{c_1}$ and $H_{c_2}$, the superconducting and normal phases coexist in the sample. As $h_{\ex}$ increases, so does the number of vortices. The vortices repeal each other, while the external magnetic field confines them inside the sample. 
\item At $H_{c_2}=O\left(\frac1{\ve^2}\right)$, the superconductivity is lost in the bulk of the sample. 
\item Between $H_{c_2}$ and $H_{c_3}$, superconductivity persists only near the boundary. 
\item Above $H_{c_3}=O\left(\frac1{\ve^2}\right)$, the applied magnetic field completely penetrates the sample and the superconductivity is lost, i.e. $u=0$.
\end{itemize}
The Ginzburg-Landau model is known to be a $\mathbb U(1)$-gauge theory. This means that all the meaningful physical quantities are invariant under the gauge transformations 
$$
u\mapsto ue^{i\phi},\quad A\mapsto A+\nabla \phi,
$$
where $\phi$ is any real-valued function in $H^2_{\loc}(\R^3)$. The Ginzburg-Landau energy and its associated free energy 
$$
F_\ve(u,A)= \frac12\int_\Omega |\nabla_A u|^2+\frac{1}{2\ve^2}(1-|u|^2)^2+|\curl A|^2
$$
are gauge invariant, as well as the density of superconducting Cooper pairs $|u|^2$, the induced magnetic field $H$, and the vorticity, defined, for any sufficiently regular configuration $(u,A)$, as
$$
\mu(u,A)=\curl (iu,\nabla_A u)+\curl A,
$$
where $(\cdot,\cdot)$ denotes the scalar product in $\C$ identified with $\R^2$ i.e. $(a,b)=\frac{\overline{a}b+a\overline{b}}2$. This quantity is the gauge-invariant version of the Jacobian determinant of $u$ and is the analogue of the vorticity of a fluid. For further physics background on the model, we refer to \cites{Tin,DeG}. 

\smallskip
The main purpose of this paper is to give a precise estimate of $H_{c_1}$ and to characterize the behavior of global minimizers of $GL_\ve$ below this value in 3D. The analysis of $H_{c_2}$ or higher applied fields requires completely different techniques. The interested reader can refer to \cites{GioPhi,FouHel,FouKach,FouKachPer} and references therein. 

\smallskip
The first critical field is (rigorously) defined by the fact that below $H_{c_1}$ global minimizers of the Ginzburg-Landau functional do not have vortices, while they do for applied fields whose strength is higher than $H_{c_1}$. In the 2D setting, Sandier and Serfaty (see \cites{Ser,SanSer1,SanSer2,SanSerBook}) provided an expansion of the first critical field, up to an error $o(1)$ as $\ve\to 0$, and rigorously characterized the behavior of global minimizers of the Ginzburg-Landau functional below and near this value. Conversely, in 3D much less is known. Very recently Baldo, Jerrard, Orlandi, and Soner \cite{BalJerOrlSon2}, via a $\Gamma$-convergence argument, provided the asymptotic leading order value of the first critical field as $\ve\to 0$ (see also \cite{BalJerOrlSon1} for related results). In short, in a uniform applied field, they proved that if $(u_\ve,A_\ve)$ minimizes $GL_\ve(u_\ve,A_\ve)$ then there exists a measure $\mu_0$ such that
$$
\frac{\mu(u_\ve,A_\ve)}{|\log \ve|}\to \mu_0 \quad \mathrm{as}\ \ve\to 0
$$
in weak sense (the precise type of convergence can be found in \cite{BalJerOrlSon2}*{Proposition 1}). Moreover, there exists a constant $H^*$ such that if $\lim_{\ve\to0}\frac{h_{\ex}}{|\log \ve|}<H^*$ then $\mu_0\equiv 0$, while $\mu_0\not\equiv 0$ if $\liminf_{\ve\to0}\frac{h_{\ex}}{|\log \ve|}>H^*$. This result
gives $H_{c_1}$ up to an error $o(|\log \ve|)$ as $\ve \to 0$ and agrees with previous work by Alama, Bronsard, and Montero \cite{AlaBroMon} in the special when $\Omega$ is a ball. An intermediate situation, when the superconducting sample is a thin shell, was treated in \cite{Con}.

\smallskip
Before stating our results, let us recall the three-dimensional $\ve$-level estimates for the Ginzburg-Landau functional provided by the author in \cite{Rom}. These tools will play a crucial role in this paper.
\begin{theorem}\label{theorem:epslevel} For any $m,n,M>0$ there exist $C,\ve_0>0$ depending only on $m,n,M,$ and $\partial\Omega$, such that, for any $\ve<\ve_0$, if $(u_\ve,A_\ve)\in H^1(\Omega,\C)\times H^1(\Omega,\R^3)$ is a configuration such that $F_\ve(u_\ve,A_\ve)\leq M|\log\ve|^m$ then there exists a polyhedral $1$-dimensional current $\nu_\ve$ such that
\begin{enumerate}[leftmargin=*,font=\normalfont]
\item $\nu_\ve /\pi$ is integer multiplicity,
\item $\partial \nu_\ve=0$ relative to $\Omega$,
\item $\mathrm{supp}(\nu_\ve)\subset S_{\nu_\ve}\subset \overline \Omega$ with $|S_{\nu_\ve}|\leq C|\log\ve|^{-q}$, where $q(m,n)\colonequals\frac32 (m+n)$,
\smallskip
\item 
$\displaystyle 
\int_{S_{\nu_\ve}}|\nabla_{A_\ve} u_\ve|^2+\frac{1}{2\ve^2}(1-|u_\ve|^2)^2+|\curl A_\ve|^2 \geq |\nu_\ve|(\Omega)\left(\log \frac1\ve-C \log \log \frac1\ve\right)-\frac{C}{|\log\ve|^n},
$
\smallskip
\item and for any $\gamma\in(0,1]$ there exists a constant $C_\gamma$ depending only on $\gamma$ and $\partial\Omega$, such that 
$$
\|\mu(u_\ve,A_\ve) -\nu_\ve\|_{C_T^{0,\gamma}(\Omega)^*}\leq C_\gamma \frac{F_\ve(u_\ve,A_\ve)+1}{|\log \ve|^{q\gamma}}.
$$
\end{enumerate}
\end{theorem}
Here and in the rest of the paper, $C_T^{0,\gamma}(\Omega)$ denotes the space of vector fields $\Phi \in C^{0,\gamma}(\Omega)$ such that $\Phi\times\nu=0$ on $\partial \Omega$, where $\nu$ is the outer unit normal to $\partial\Omega$. The symbol $^*$ denotes its dual space.

\subsection{Main results}
Throughout this article we assume that $H_{\ex}\in L_{\loc}^2(\R^3,\R^3)$ is such that $\diver H_{\ex}=0$ in $\R^3$. In particular, we deduce that there exists a vector-potential $A_{\ex}\in H^1_{\loc}(\R^3,\R^3)$ such that 
$$
\curl A_{\ex}=H_{\ex}, \ \diver A_{\ex}=0\ \mathrm{in}\ \R^3\quad \mathrm{and}\quad A_{\ex}\cdot \nu=0\ \mathrm{on}\ \partial\Omega.
$$
Let us define $H_{0,\ex}=h_{\ex}^{-1}H_{\ex}$ and assume that this vector field is H\"older continuous in $\Omega$ with H\"older exponent $\beta\in (0,1]$ and H\"older norm bounded independently of $\ve$. In particular, note that $\|H_{0,\ex}\|_{L^2(\Omega,\R^3)}=1$. It is also convenient to set $A_{0,\ex}=h_{\ex}^{-1}A_{\ex}$.

We remark that the divergence-free assumption on the applied magnetic field is in accordance with the fact that magnetic monopoles do not exist in Maxwell's electromagnetism theory.

\smallskip 
The natural space for the minimization of $GL_\ve$ in 3D is $H^1(\Omega,\C)\times [A_{\ex}+H_{\curl}]$, where
$$
H_{\curl}\colonequals \{A\in H^1_{\loc}(\R^3,\R^3) \ | \ \curl A\in L^2(\R^3,\R^3)\}.
$$
Let us also introduce the homogeneous Sobolev space $\dot H^1(\R^3,\R^3)$, which is defined as the completion of $C_0^\infty (\R^3,\R^3)$ with respect to the norm $\|\nabla (\, \cdot \, ) \|_{L^2(\R^3,\R^3)}$. We observe that, by Sobolev embedding, there exists a constant $C>0$ such that 
\begin{equation}\label{space1}
\|A\|_{L^6(\R^3,\R^3)}\leq C \|\nabla A\|_{L^2(\R^3,\R^3)}
\end{equation}
for any $A\in \dot H^1(\R^3,\R^3)$. Moreover, by \cite{KozSoh}*{Proposition 2.4}, we have 
$$
\dot H^1(\R^3,\R^3)=\{A\in L^6(\R^3,\R^3)\ | \ \nabla A \in L^2(\R^3,\R^3)\}.
$$
It is also convenient to define the subspace 
$$
\dot H^1_{\diver=0}\colonequals \{A\in \dot H^1(\R^3,\R^3) \ | \ \diver A=0\ \mathrm{in}\ \R^3\}.
$$
In this subspace, one has
\begin{equation}\label{space2}
\|A\|_{\dot H^1_{\diver=0}}\colonequals \|\nabla A\|_{L^2(\R^3,\R^3)}= \|\curl A\|_{L^2(\R^3,\R^3)}.
\end{equation}

\medskip
Let us now define a special vortexless configuration that turns out to be a good approximation of the so-called \emph{Meissner solution}, i.e. the vortexless global minimizer of the Ginzburg-Landau energy below the first critical field, which, as we shall see, is unique up to a gauge transformation. By recalling that any vector field $A\in H^1(\Omega,\R^3)$ can be decomposed as (see Lemma \ref{lemma:decompositionOmega})
$$
\left\lbrace 
\begin{array}{rcll}
A&=&\curl B_A + \nabla \phi_A&\mathrm{in}\ \Omega\\
B_A\times \nu&=&0&\mathrm{on}\ \partial\Omega\\
\nabla\phi_A\cdot \nu&=&A\cdot \nu &\mathrm{on}\ \partial\Omega,
\end{array}\right.
$$
we consider the pair $(u_0,h_{\ex}A_0)$, where $u_0=e^{ih_{\ex}\phi_{A_0}}$ and $A_0$ is the unique minimizer (in a suitable space) of the functional 
$$
J(A)=\frac12\int_\Omega|\curl B_A|^2+\frac12\int_{\R^3}|\curl(A-A_{0,ex})|^2.
$$
This special configuration satisfies the following properties:
\begin{itemize}
\item $GL_\ve(u_0,h_{\ex}A_0)=h_{\ex}^2J(A_0)$ and $|u_0|=1$ in $\Omega$.
\item $H_0=\curl A_0$ satisfies the usually called \emph{London equation}
$$
\curlcurl(H_0-H_{0,\ex})+H_0\chi_\Omega=0\quad \mathrm{in}\  \R^3,
$$
where $\chi_\Omega$ denotes the characteristic function of $\Omega$.
\item The divergence-free vector field $B_0=B_{A_0}\in C_T^{2,\beta}(\Omega,\R^3)$ satisfies 
$$
\left\{
\begin{array}{rcll}
-\Delta B_0+B_0&=&H_{0,\ex}&\mathrm{in}\ \Omega\\
B_0\times \nu&=&0&\mathrm{on}\ \partial\Omega.
\end{array}\right.
$$
This vector field is the analog of the function $\xi_0$, considered by Sandier and Serfaty in the analysis of the first critical field in 2D (see \cites{Ser,Ser2,SanSer1,SanSer2}). We shall see that $B_0$ plays an important role in our 3D analysis.
\end{itemize}
In addition, this pair allows us to split the Ginzburg-Landau energy of a given configuration $(u,A)$. More precisely, by writing $u'=u_0^{-1}u$ and $A'=A-h_{\ex}A_0$, one can prove that (see Proposition \ref{prop:energysplitting})
$$
GL_\ve(u,A)=h_{\ex}^2J(A_0)+F_\ve(u',A')+\frac12\int_{\R^3\setminus \Omega}|\curl A'|^2-h_{\ex}\int_\Omega \mu(u',A')\wedge B_0+R_0,
$$
where $R_0=o(1)$, in particular, when $h_{\ex}$ is bounded above by a negative power of $|\log\ve|$. Let me emphasize that one of the achievement of this paper is to find the right pair $(u_0,h_{\ex}A_0)$ to split the energy, which then allows to implement (almost) the same strategies as in 2D.

By combining this splitting with the optimal $\ve$-level estimates of Theorem \ref{theorem:epslevel}, we find
\begin{multline*}
GL_\ve(u,A)\geq h_{\ex}^2J(A_0)+\frac12|\nu_\ve'|(\Omega)\left(\log\frac1\ve -C\log \log \frac1\ve\right)\\+\frac12\int_{\R^3\setminus \Omega}|\curl A'|^2-h_{\ex}\int_\Omega \nu_\ve'\wedge B_0+o(1),
\end{multline*}
where $\nu_\ve'$ denotes the $1$-current associated to $(u',A')$ by Theorem \ref{theorem:epslevel}. By construction of $\nu_\ve'$ (see \cite{Rom}*{Section 5.2}), we can write
$$
\nu_\ve'=\sum_{i\in I_\ve}2\pi \Gamma_i^\ve,
$$
where the sum is understood in the sense of currents, $I_\ve$ is a finite set of indices, and $\Gamma_i^\ve$ is an oriented Lipschitz curve in $\Omega$ with multiplicity $1$. Each of these curves, which are non-necessarily distinct, does not self intersect and is either a loop contained in $\Omega$ or has two different endpoints on $\partial \Omega$. We will denote by $X$ the class of Lipschitz curves, seen as $1$-currents, described here.

Inserting this expression in the previous inequality, allows us to heuristically derive the leading order of the first critical field:
$$
H_{c_1}^0\colonequals \frac1{2\|B_0\|_*}|\log \ve|,
$$
where
\begin{equation}\label{B_0*}
\|B_0\|_*\colonequals \sup\limits_{\Gamma \in X}\frac1{|\Gamma|(\Omega)}\int_\Omega \Gamma\wedge B_0\footnote{The notation used here is explained in the preliminaries (see Section \ref{sec:preliminaries}).}.
\end{equation}
In Proposition \ref{B_0ball}, we compute this value in a special case. 

\smallskip
We may now state our first result, that characterizes the behavior of global minimizers of $GL_\ve$ below $H_{c_1}^0$. In the 2D setting, an analogous result was proved by Sandier and Serfaty (see \cite{SanSer1}*{Theorem 1}).
\begin{theorem}\label{theorem:belowHc1}
There exist constants $\ve_0,K_0>0$ such that for any $\ve<\ve_0$ and $h_{\ex}\leq H_{c_1}^0-K_0\log |\log\ve|$, the global minimizers $(u_\ve,A_\ve)$ of $GL_\ve$ in $H^1(\Omega,\C)\times [A_{\ex}+H_{\curl}]$ are vortexless configurations that satisfy
$$
h_{\ex}^2J(A_0)+o(1)\leq GL_\ve(u_\ve,A_\ve)\leq h_{\ex}^2J(A_0)\quad \mathrm{and}\quad \|1-|u_\ve|\|_{L^\infty(\Omega,\C)}=o(1).
$$
\end{theorem}
It is important to mention that in the proof of this theorem we use the fact that solutions of the Ginzburg-Landau equations (see Section \ref{Section:GLeq}), in the Coulomb gauge, satisfy a clearing-out result proved by Chiron \cite{Chi}. Roughly speaking, this states that if the energy of a solution in a ball (with center in $\overline \Omega$) intersected with $\Omega$ is sufficiently small, then $|u|$ is uniformly away from $0$ in a ball of half radius intersected with $\Omega$. The proof given by Chiron relies on monotonicity formulas, and is very much inspired by previous work by Bethuel, Orlandi, and Smets \cite{BetOrlSme}. The interested reader can refer to \cites{Riv,LinRiv1,LinRiv2,BetBreOrl,SanSha} for results in the same spirit.

\medskip
Our second result provides bounds from above and below for the first critical field in 3D. 
\begin{theorem}\label{theorem:estimateHc1}
There exist constants $\ve_0,K_0,K_1>0$ such that for any $\ve<\ve_0$ we have 
$$
H_{c_1}^0-K_0\log |\log \ve|\leq H_{c_1}\leq H_{c_1}^0+K_1.
$$
\end{theorem}
In particular, these inequalities show that indeed $H_{c_1}^0$ is the leading order of $H_{c_1}$ as $\ve\to 0$. Of course this agrees with the previously mentioned result by Baldo, Jerrard, Orlandi, and Soner within their framework (i.e. when $H_{0,\ex}$ is taken to be a fixed unit vector). The author strongly believes that, as $\ve\to 0$,
$$
H_{c_1}=H_{c_1}^0+O(1).
$$
To prove this result, one needs to avoid the uncertainty of order $O(\log|\log \ve|)$ in the lower bound for $H_{c_1}$ of Theorem \ref{theorem:belowHc1}. To accomplish this, it is crucially important to characterize, near the first critical field, the behavior of the vorticity $\mu(u,A)$ of global minimizers of $GL_\ve$. We plan to address this question in future work.

\medskip
Our next result shows that beyond the first critical field there exists a locally minimizing vortexless configuration. A similar result was proved by Serfaty in 2D (see \cite{Ser2}*{Theorem 1}).

\begin{theorem}\label{theorem:meissner} Let $\alpha\in \left(0,\frac13\right)$. There exists $\ve_0>0$ such that, for any $\ve<\ve_0$, if $h_{\ex}\leq \ve^{-\alpha}$ then there exists a vortexless configuration $(u_\ve,A_\ve)=(u_0u_\ve',h_{\ex}A_0+A_\ve')\in H^1(\Omega,\C)\times [A_{\ex}+\dot H^1_{\diver=0}]$ with $A_\ve'\cdot \nu=0$ on $\partial\Omega$, which locally minimizes $GL_\ve$ in $H^1(\Omega,\C)\times [A_{\ex}+H_{\curl}]$. In addition, it satisfies the following properties:
\begin{enumerate}[leftmargin=*,font=\normalfont]
\item $h_{\ex}^2J(A_0)+o(1)\leq GL_\ve(u_\ve,A_\ve)\leq h_{\ex}^2J(A_0)$ and $\|1-|u_\ve|\|_{L^\infty(\Omega,\C)}=o(1)$.
\item The configuration $(u_\ve',A_\ve')$ satisfies
$$
\inf_{\theta\in[0,2\pi]}\|u_\ve'-e^{i\theta}\|_{H^1(\Omega,\C)}+\|A_\ve'\|_{\dot H^1_{\diver=0}}\to 0\quad \mathrm{as} \ \ve\to0.
$$
\item Up to a gauge transformation, $(u_\ve,A_\ve)$ converges to $(u_0,h_{\ex}A_0)$. More precisely, we have
$$
\inf_{\theta\in[0,2\pi]}\|u_\ve-e^{i\theta}u_0\|_{H^1(\Omega,\C)}+\|A_\ve-h_{\ex}A_0\|_{\dot H^1_{\diver=0}}\to 0\quad \mathrm{as} \ \ve\to0.
$$
\end{enumerate}
\end{theorem}
Let us point out that in Remark \ref{remark:condalpha} we explain why we require $\alpha<\frac13$.

\medskip
Our last result concerns the uniqueness, up to a gauge transformation, of locally minimizing vortexless configurations.
\begin{theorem}\label{theorem:uniqueness}
Let $\alpha,c\in (0,1)$. There exists $\ve_0>0$ such that, for any $\ve<\ve_0$, if $h_{\ex}\leq \ve^{-\alpha}$ then a configuration $(u,A)$ which locally minimizes $GL_\ve$ in $H^1(\Omega,\C)\times [A_{\ex}+H_{\curl}]$ and satisfies $|u|\geq c$ and $F_\ve(u',A')\leq \ve^{1+\de}$ for some $\de>0$, is unique up to a gauge transformation.
\end{theorem}

\begin{remark}
The assumption that $F_\ve(u',A')\leq \ve^{1+\de}$ for some $\de>0$ plays a crucial role in the proof of this result. In Proposition \ref{prop:estimteFree}, we prove that if $\alpha\in\left(0,\frac14\right)$ then this condition is implied by the other assumptions of this theorem provided that $GL_\ve(u,A)\leq GL_\ve(u_0,h_{\ex}A_0)=h_{\ex}^2J(A_0)$, i.e. uniqueness holds without assuming that $F_\ve(u',A')\leq \ve^{1+\de}$ for some $\de>0$ if the Ginzburg-Landau energy of the vortexless local minimizer is below the energy of $(u_0,h_{\ex}A_0)$. 
We observe that this condition is satisfied by the locally minimizing solution of Theorem \ref{theorem:meissner}.

Let us also note that if $\alpha\geq \frac14$ then the strategy of the proof of Proposition \ref{prop:estimteFree} fails. For this reason, we are able to guaranty the uniqueness of the locally minimizing vortexless configuration of Theorem \ref{theorem:meissner} only if $\alpha<\frac14$.

Finally, let us emphasize that this uniqueness result allows to conclude that the locally minimizing configuration of Theorem \ref{theorem:meissner} is, indeed, up to a gauge transformation, the unique global minimizer of the Ginzburg-Landau energy below the first critical field. Therefore Theorem \ref{theorem:meissner}, in particular, provides a detailed characterization of the behavior of the Meissner solution.
\end{remark}
Thus, we prove that below the first critical field, up to a gauge transformation, the Meissner solution is the unique global minimizer of $GL_\ve$. Beyond this value, at least up to $h_{\ex}=o(\ve^{-\frac13})$, a Meissner-type solution continues to exists as a local minimizer of the Ginzburg-Landau energy. This solution is unique, up to a gauge transformation, at least up to $h_{\ex}=o(\ve^{-\frac14})$. Since this branch of vortexless solutions remains stable, in the process of raising $h_{\ex}$ vortices should not appear at $H_{c_1}$, but rather at a critical value of $h_{\ex}$ called the \emph{superheating field} $H_{\mathrm{sh}}$, at which the Meissner-type solution becomes unstable. It is expected that $H_{\mathrm{sh}}=O(\ve^{-1})$. The interested reader can refer to \cite{Xia} and references therein for further details.

\subsection*{Outline of the paper} The rest of the paper is organized as follows. In Section \ref{sec:preliminaries} we introduce some basic quantities and notation, describe two Hodge-type decompositions, and present some classical results in Ginzburg-Landau theory. In Section \ref{sec:below} we define the approximation of the Meissner solution, split the Ginzburg-Landau energy, and prove Theorem \ref{theorem:belowHc1}. In Section \ref{sec:Hc1} we present the proof of Theorem \ref{theorem:estimateHc1} and compute $\|B_0\|_*$ in a special case. Section \ref{sec:Meissner} contains the proof of Theorem \ref{theorem:meissner} and Section \ref{sec:Uniqueness} the proof of Theorem \ref{theorem:uniqueness}. Appendix \ref{sec:appendix} is devoted to prove some improved estimates for locally minimizing configurations, that allow to obtain the uniqueness of the Meissner-type solution of Theorem \ref{theorem:meissner} for $\alpha<\frac14$, as a consequence of Theorem \ref{theorem:uniqueness}.

\section{Preliminaries}\label{sec:preliminaries}
\subsection{Some definitions and notation}
We define the superconducting current of a pair $(u,A)\in H^1(\Omega,\C)\times H^1(\Omega,\R^3)$ as the $1$-form
$$
j(u,A)=(iu,d_A u)=\sum_{k=1}^3 (iu,\partial_k u-iA_ku)dx_k.
$$
It is related to the vorticity $\mu(u,A)$ of a configuration $(u,A)$ through
$$
\mu(u,A)=dj(u,A)+dA.
$$
This quantity can be seen as a $1$-current, which is defined through its action on $1$-forms by the relation
$$
\mu(u,A)(\phi)=\int_\Omega \mu(u,A)\wedge \phi.
$$
We recall that the boundary of a $1$-current $T$ relative to a set $\Theta$, is the $0$-current $\partial T$ defined by
$$
\partial T(\phi)=T(d\phi)
$$
for all smooth compactly supported $0$-form $\phi$ defined in $\Theta$. In particular, $\mu(u,A)$ has zero boundary relative to $\Omega$. We denote by $|T|(\Theta)$ the mass of a $1$-current $T$ in $\Theta$.

\subsection{Hodge-type decompositions}
Next, we provide a decomposition of vector fields in $H_{\curl}$.
\begin{lemma}\label{lemma:decompositionR3}
Every vector field $A\in H_{\curl}$ can be decomposed as
$$
A=\curl \BB +\nabla \Phi,
$$
where $\BB,\curl \BB\in \dot H^1_{\diver=0}$ and $\Phi\in H^2_{\loc}(\R^3)$.
\end{lemma}
\begin{proof}
First, let us observe that there exists a function $\Phi_1\in H^2_{\loc}(\R^3,\R^3)$ such that
$$
\Delta \Phi_1 = \diver A \in L^2_{\loc}(\R^3,\R^3).
$$
Second, we consider the problem
$$
\left\{
\begin{array}{rcl}
\curlcurl B &=& \curl A \in L^2(\R^3,\R^3)\\
\diver B&=&0.
\end{array}\right.
$$
By observing that $\curlcurl B =-\Delta B$, \cite{KozSoh}*{Theorem 1} provides the existence of a solution $\BB\in \dot H^1_{\diver=0}$ to this problem such that $\curl \BB\in \dot H^1_{\diver=0}$.

Finally, by noting that 
$$
\curl (A-\nabla \Phi_1-\curl \BB)=\diver (A-\nabla \Phi_1-\curl \BB)=0,
$$
we deduce that 
$$
A-\nabla \Phi_1-\curl \BB=\nabla \Phi_2
$$
for some harmonic function $\Phi_2\in H^2_{\loc}(\R^3,\R^3)$. By writing $\Phi=\Phi_1+\Phi_2$, we obtain the result.
\end{proof}

We now recall a decomposition of vector fields in $H^1(\Omega,\R^3)$. The proof of this result can be found in \cite{BetBreOrl}*{Appendix A}.
\begin{lemma}\label{lemma:decompositionOmega}
There exists a constant $C=C(\Omega)$ such that for every $A\in H^1(\Omega,\R^3)$ there exist a unique vector field $B_A\in \{B\in H^2(\Omega,\R^3)\ | \ \diver B=0\ \mathrm{in}\ \Omega\}$ and a unique function $\phi_A\in \{\phi \in H^2(\Omega)\ | \ \int_\Omega \phi_A=0\}$ satisfying 
$$
\left\lbrace 
\begin{array}{rcll}
A&=&\curl B_A + \nabla \phi_A&\mathrm{in}\ \Omega\\
B_A\times \nu&=&0&\mathrm{on}\ \partial\Omega\\
\nabla\phi_A\cdot \nu&=&A\cdot \nu &\mathrm{on}\ \partial\Omega
\end{array}\right.
$$
and
$$
\|B_A\|_{H^2(\Omega,\R^3)}+\|\phi_A\|_{H^2(\Omega)}\leq C\|A\|_{H^1(\Omega,\R^3)}.
$$
\end{lemma}

\subsection{Ginzburg-Landau equations}\label{Section:GLeq}
\begin{definition}[Critical point of $GL_\ve$] We say that $(u,A)\in H^1(\Omega,\C)\times [A_{\ex}+H_{\curl}]$ is a critical point of $GL_\ve$ if for every  smooth and compactly supported configuration $(v,B)$ we have 
$$
\frac d{dt}GL_\ve(u+tv,A+tB)|_{t=0}=0.
$$
\end{definition}
We now present the Euler-Lagrange equations satisfied by critical points of $GL_\ve$. This is a well-known result, but for the sake of completeness we prove it here.
\begin{proposition}[Ginzburg-Landau equations]
If $(u,A)\in H^1(\Omega,\C)\times [A_{\ex}+H_{\curl}]$ is a critical point of $GL_\ve$ then $(u,A)$ satisfies the system of equations
\begin{equation}\label{GLeq}
\tag{GL}
\left\lbrace 
\begin{array}{rcll}
-(\nabla_A)^2u&=&\displaystyle\frac1{\ve^2}u(1-|u|^2)&\mathrm{in}\ \Omega\\
\curl(H-H_{\ex})&=&(iu,\nabla_A u)\chi_\Omega &\mathrm{in}\ \R^3\\
\nabla_A u\cdot \nu&=&0&\mathrm{on}\ \partial \Omega\\
\lbrack H-H_{\ex}\rbrack\times\nu&=&0&\mathrm{on}\ \partial\Omega, 
\end{array}
\right.
\end{equation}
where $\chi_\Omega$ is the characteristic function of $\Omega$, $[\, \cdot \, ]$ denotes the jump across $\partial\Omega$, $\nabla_A u\cdot \nu=\sum_{j=1}^3 (\partial_ju-iA_ju)\nu_j$, and the covariant Laplacian $(\nabla_A)^2$is defined by
$$
(\nabla_A)^2u=(\diver -iA\cdot )\nabla_Au.
$$
\end{proposition}
\begin{proof}
We have
$$
\frac d{dt}GL_\ve(u+tv,A)|_{t=0}=\int_\Omega (\nabla_A u,\nabla_A v)-\frac1{\ve^2}\int_\Omega (u,v)(1-|u|^2).
$$
By noting that
$$
(\nabla_A u, \nabla_A v)=\diver (\nabla_A u,v)-((\nabla_A)^2u,v),
$$
where $(\nabla_A u,v)=((\partial_1u-iA_1u,v),(\partial_2u-iA_2u,v),(\partial_3u-iA_3u,v))$, and by integrating by parts, we obtain
$$
\frac d{dt}GL_\ve(u+tv,A)|_{t=0}= \int_{\partial\Omega} (\nabla_A u \cdot \nu,v)-\int_\Omega ((\nabla_A)^2u,v)-\frac1{\ve^2}\int_\Omega (u,v)(1-|u|^2).
$$
Since this is true for any $v$, we find
$$
(\nabla_A)^2u=\frac1{\ve^2}u(1-|u|^2)\ \mathrm{in}\ \Omega\quad \mathrm{and}\quad \nabla_A u\cdot \nu=0\ \mathrm{on}\ \partial\Omega.
$$
On the other hand, we have
$$
\frac d{dt}GL_\ve(u,A+tB)|_{t=0}=-\int_\Omega (i Bu,\nabla_A u)+\int_{\R^3}(H-H_{\ex})\cdot \curl B.
$$
By integration by parts, we get
$$
\frac d{dt}GL_\ve(u,A+tB)|_{t=0}=-\int_\Omega (iu,\nabla_A u)\cdot B+\int_{\R^3}\curl(H-H_{\ex})\cdot B.
$$
We deduce that 
$$
\curl(H-H_{\ex})=(iu,\nabla_A u)\chi_\Omega \quad \mathrm{in}\ \R^3.
$$
By testing this equation against $B$ and by integrating by parts over $\Omega$, we find
$$
\int_\Omega (H-H_{\ex})\cdot \curl B-\int_{\partial \Omega}((H-H_{\ex})\times \nu)\cdot B-\int_\Omega (iu,\nabla_A u)\cdot B=0.
$$
Now, by integrating by parts over $\R^3\setminus \Omega$, we get
$$
\int_{\R^3 \setminus \Omega} (H-H_{\ex})\cdot \curl B+\int_{\partial (\R^3\setminus \Omega)}((H-H_{\ex})\times \nu)\cdot B=0.
$$
Thus
$$
\int_{\partial \Omega}([H-H_{\ex}]\times \nu)\cdot B=0,
$$
which implies that $[H-H_{\ex}]\times \nu=0$ on $\partial\Omega$.
\end{proof}

\begin{remark}
By taking the curl of the second Ginzburg-Landau equation, we find
\begin{equation}\label{EqVorticity}
\curlcurl (H-H_{\ex})+H\chi_\Omega=\mu(u,A)\chi_\Omega,
\end{equation}
in the sense of currents.
We will come back to this equation later on.
\end{remark}

\subsection{Minimization of \texorpdfstring{$GL_\ve$}{the Ginzburg-Landau energy}}
\begin{proposition}\label{prop:minimization}
The minimum of $GL_\ve$ over $H^1(\Omega,\C)\times [A_{\ex}+H_{\curl}]$ is achieved.
\end{proposition}
\begin{proof}
Let $\{(\tilde u_n,\tilde A_n)\}_n$ be a minimizing sequence for $GL_\ve$ in $H^1(\Omega,\C)\times [A_{\ex}+H_{\curl}]$. Lemma \ref{lemma:decompositionR3} yields a gauge transformed sequence $\{(u_n,A_n)\}_n$ such that $A_n\in [A_{\ex}+\dot H^1_{\diver=0}]$. In particular, we have that $GL_\ve(\tilde u_n,\tilde A_n)=GL_\ve(u_n,A_n)$ and
$$
\|\nabla (A_n-A_{\ex})\|_{L^2(\R^3,\R^3)}=\|\curl (A_n-A_{\ex})\|_{L^2(\R^3,\R^3)}.
$$
Using the bound $GL_\ve(u_n,A_n)\leq C$, where $C$ is independent of $n$, we find that 
$$
\|1-|u_n|^2\|_{L^2(\Omega,\C)},\ \|\nabla_{A_n}u_n\|_{L^2(\Omega,\C^3)},\ \mathrm{and}\quad \|\curl (A_n-A_{\ex})\|_{L^2(\R^3,\R^3)}
$$
are bounded independently of $n$. Therefore, by recalling \eqref{space1}, we deduce that $A_n-A_{\ex}$ is bounded in $\dot H^1(\R^3,\R^3)$. 
Because $\{u_n\}_n$ is bounded in $L^4(\Omega)$ we find that $\{iA_nu_n\}_n$ is bounded in $L^2(\Omega,\C^3)$. By noting that $\nabla u_n=\nabla_{A_n}u_n+iA_nu_n$, we conclude that $u_n$ is bounded in $H^1(\Omega,\C)$.

We may then extract a subsequence, still denoted $\{(u_n,A_n)\}_n$, such that $\{(u_n,A_n-A_{\ex})\}_n$ converges to some $(u,A-A_{\ex})$ weakly in $H^1(\Omega,\C)\times \dot H_{\diver=0}$ and, by compact Sobolev embedding, strongly in every $L^q(\Omega,\C)\times L^q(\Omega,\R^3)$ for $q<6$. 

Let us now show that $(u,A)$ is a minimizer of $GL_\ve$. By strong $L^4(\Omega,\C)$ convergence, 
$$
\liminf_n \|1-|u_n|^2\|_{L^2(\Omega,\C)}=\|1-|u|^2\|_{L^2(\Omega,\C)}.
$$
Also, by weak $\dot H^1(\R^3,\R^3)$ convergence, we have
\begin{align*}
\liminf_n\|\curl (A_n-A_{\ex})\|_{L^2(\R^3,\R^3)}&=\liminf_n\|\nabla (A_n-A_{\ex})\|_{L^2(\R^3,\R^3)}\\ 
&\geq \|\nabla (A-A_{\ex})\|_{L^2(\R^3,\R^3)}=\|\curl(A-A_{\ex})\|_{L^2(\R^3,\R^3)}.
\end{align*}
Moreover, standard arguments show that 
\begin{align*}
\liminf_n \|\nabla_{A_n}u_n\|_{L^2(\Omega,\C^3)}^2&=\liminf_n \|\nabla u_n\|^2_{L^2(\Omega,\C^3)}-2\int_\Omega (\nabla u_n,iA_nu_n)+\|A_nu_n\|_{L^2(\Omega,\C^3)}
\\ &\geq  \|\nabla_{A}u\|_{L^2(\Omega,\C^3)}.
\end{align*}
Hence
$$
\liminf_n GL_\ve(u_\ve,A_\ve)\geq GL_\ve(u,A).
$$
\end{proof}
\section{Global minimizers below \texorpdfstring{$H_{c_1}^0$}{Hc\textoneinferior\textzerosuperior}}\label{sec:below}
\subsection{An approximation of the Meissner solution}
Next, we find a configuration $(u_0,h_{\ex}A_0)$ with $|u_0|=1$ and which satisfies \eqref{EqVorticity} with zero right-hand side. As mentioned in the introduction, this turns out to be a good approximation of the Meissner solution, the vortexless configuration which minimizes $GL_\ve$ below the first critical field. 

\smallskip
Let us consider a configuration of the form $(e^{i \phi_0},h_{\ex}A_0)$ with $\phi_0\in H^2(\Omega)$ and $A_0\in A_{0,\ex}+\dot H^1_{\diver=0}$. Observe that, by using Lemma \ref{lemma:decompositionOmega} and by letting $u_0\colonequals e^{i \phi_0}$, we have
\begin{align*}
GL_\ve(u_0,h_{\ex}A_0)=&\frac12 \int_\Omega |\nabla \phi_0-h_{\ex}(\curl B_{A_0}+\nabla\phi_{A_0})|^2+\frac12\int_{\R^3}|h_{\ex}\curl A_0-H_{\ex}|^2\\
=&\frac12 \int_\Omega |\nabla (\phi_0-h_{\ex}\phi_{A_0})|^2+ h_{\ex}^2|\curl B_{A_0}|^2+\frac{h_{\ex}^2}2\int_{\R^3}|\curl (A_0-A_{0,\ex})|^2.
\end{align*}
By choosing $\phi_0=h_{\ex}\phi_{A_0}$, we obtain
$$
GL_\ve(u_0,h_{\ex}A_0)=\frac{h_{\ex}^2}2\int_\Omega|\curl B_{A_0}|^2+\frac{h_{\ex}^2}2\int_{\R^3}|\curl (A_0-A_{0,\ex})|^2\equalscolon h_{\ex}^2 J(A_0).
$$
We let $A_0$ to be the minimizer of $J$ in the space $\left(A_{0,\ex}+\dot H^1_{\diver=0},\|\, \cdot \, \|_{\dot H^1_{\diver=0}}\right)$, whose existence and uniqueness follows by noting that $J$ is continuous, coercive, and strictly convex in this Hilbert space (recall \eqref{space1} and \eqref{space2}). We also let $H_0=\curl A_0$ and here and in the rest of the paper we use the notation $B_0\colonequals B_{A_0}$. 

One can easily check that, for any $A\in \dot H^1_{\diver=0}$, we have 
$$
\int_\Omega \curl B_0\cdot \curl B_A + \int_{\R^3} (H_0-H_{0,\ex})\cdot \curl A=0
$$
Because
$$
\int_\Omega \curl B_0\cdot \nabla \phi_A=\int_\Omega B_0 \cdot \curl \nabla \phi_A-\int_{\partial \Omega} (B_0\times \nu)\phi_A=0,
$$
we have
\begin{equation}\label{A0}
\int_\Omega \curl B_0\cdot A + \int_{\R^3} (H_0-H_{0,\ex})\cdot \curl A=0
\end{equation}
Moreover, Lemma \ref{lemma:decompositionR3} implies that this equality also holds for any $A\in H_{\curl}$.

\medskip
Let us observe that, for any $A\in C_0^\infty (\R^3,\R^3)$, by integration by parts, we have
$$
\int_\Omega \curl B_0\cdot A+ \int_{\R^3} \curl(H_0-H_{0,\ex})\cdot A=0.
$$
Therefore $A_0$ satisfies the Euler-Lagrange equation
\begin{equation}\label{EulerLagrangeA0}
\curl(H_0-H_{0,\ex})+\curl B_0\chi_\Omega=0\quad \mathrm{in}\ \R^3.
\end{equation}
In addition, it is easy to see that the boundary condition $[H_0-H_{0,\ex}]=0$ on $\partial\Omega$ holds. 

By taking the curl of the previous equation, we find 
$$
\curlcurl(H_0-H_{0,\ex})+H_0\chi_\Omega=0\quad \mathrm{in}\  \R^3,
$$
namely (up to multiplying by $h_{\ex}$) \eqref{EqVorticity} with $\mu(u_0,A_0)=0$. 

\medskip
On the other hand, by integration by parts, for any vector field $B\in C_0^\infty(\Omega,\R^3)$, we have
$$
\int_\Omega B_0\cdot \curl B+\int_\Omega (H_0-H_{0,\ex})\cdot \curl B=0.
$$
Besides, for any function $\phi \in C_0^\infty(\Omega)$, we have
$$
\int_\Omega (B_0+(H_0-H_{0,\ex}))\cdot \nabla \phi =-\int_\Omega \diver (B_0+(H_0-H_{0,\ex})) \phi =0
$$
Then, given any vector field $A\in C_0^\infty(\Omega,\R^3)$, by taking $B=B_A$ and $\phi=\phi_A$ in the previous equalities, we find
$$
\int_\Omega (B_0+(H_0-H_{0,\ex}))\cdot (\curl B_A+\nabla \phi_A)=\int_\Omega (B_0+(H_0-H_{0,\ex}))\cdot A=0.
$$
Hence, the free-divergence vector field $B_0$ weakly solves the problem
\begin{equation}\label{eqB_0}
\left\{
\begin{array}{rcll}
-\Delta B_0+B_0&=&H_{0,\ex}&\mathrm{in}\ \Omega\\
B_0\times \nu&=&0&\mathrm{on}\ \partial\Omega.
\end{array}\right.
\end{equation}
\begin{remark}\label{remark:B_0}
Since we assume that $\|H_{0,\ex}\|_{C^{0,\beta}(\Omega,\R^3)}<C$, by standard elliptic regularity theory, we deduce that $B_0\in C_T^{2,\beta}(\Omega,\R^3)$ with $\|B_0\|_{C_T^{2,\beta}(\Omega,\R^3)}\leq C$ for some constant independent of $\ve$. In addition, if the applied field is taken to be uniform in $\Omega$, i.e. if $H_{0,\ex}$ is a fixed unit vector in $\Omega$, then $B_0$ depends on the domain $\Omega$ only.
\end{remark}

\subsection{Energy-splitting}
Next, by using the approximation of the Meissner solution, we present a splitting of $GL_\ve$.
\begin{proposition}\label{prop:energysplitting} For any $(u,A)\in H^1(\Omega,\C)\times [A_{\ex}+H_{\curl}]$, letting $u=u_0u'$ and $A=h_{\ex}A_0+A'$, where $(u_0,h_{\ex}A_0)$ is the approximation of the Meissner solution, we have
\begin{equation}\label{Energy-Splitting}GL_\ve(u,A)=h_{\ex}^2J(A_0)+F_\ve(u',A')+\frac12\int_{\R^3\setminus \Omega}|\curl A'|^2-h_{\ex}\int_\Omega \mu(u',A')\wedge B_0+R_0,
\end{equation}
where $F_\ve(u',A')$ is the free energy of the configuration $(u',A')\in H^1(\Omega,\C)\times H_{\curl}$, i.e.
$$
F_\ve(u',A')=\frac12 \int_\Omega |\nabla_{A'}u'|^2+\frac1{2\ve^2}(1-|u'|^2)^2+|\curl A'|^2
$$
and
$$ 
R_0=\frac{h_{\ex}^2}2 \int_\Omega (|u|^2-1)|\curl B_0|^2.
$$
In particular, $|R_0|\leq C\ve h_{\ex}^2E_\ve(|u|)^{\frac12}$ with $\displaystyle E_\ve(|u|)=\frac12 \int_\Omega |\nabla |u||^2+\frac1{2\ve^2}(1-|u|^2)^2$.
\end{proposition}
\begin{proof}
One immediately checks that $A'\in H_{\curl}$. Since $u'=u_0^{-1}u=e^{-ih_{\ex}\phi_0}u$ and $\phi_0\in H^2(\Omega,\C)$, by Sobolev embedding we deduce that $u'\in H^1(\Omega,\C)$. 

Writing $u=u_0u'$ and $A=h_\ex A_0+A'$ and plugging them into $GL_\ve(u,A)$, we obtain
\begin{multline*}
GL_\ve(u,A)=
\frac12 \int_\Omega |\nabla_{A'}u'-ih_{\ex}\curl B_0u'|^2+\frac1{2\ve^2}(1-|u'|^2)^2\\+\frac12\int_{\R^3}|\curl A'+h_\ex (H_0-H_{0,\ex})|^2.
\end{multline*}
By expanding the square terms, we get
\begin{align*}
GL_\ve(u,A)=&\frac12 \int_\Omega |\nabla_{A'}u'|^2+h_{\ex}^2|\curl B_0|^2|u'|^2-2h_{\ex}(\nabla_{A'}u',iu')\cdot\curl B_0+\frac1{2\ve^2}(1-|u'|^2)^2 \\
+&\frac12 \int_{\R^3} |\curl A'|^2+h_\ex^2 |H_0-H_{0,\ex}|^2+2h_\ex \curl A'\cdot (H_0-H_{0,\ex}).
\end{align*}
Observe that, by \eqref{A0}, we have
$$
\int_{\R^3} \curl A'\cdot (H_0-H_{0,\ex})=-\int_\Omega A'\cdot\curl B_{0}.
$$
Therefore, grouping terms and writing $|u'|^2$ as $1+(|u'|^2-1)$, we find
$$
GL_\ve(u,A)=h_{\ex}^2J(A_0)+F_\ve(u',A')+\frac12 \int_{\R^3\setminus\Omega}|\curl A'|^2-h_{\ex}\int_\Omega (j(u',A')+A')\cdot \curl B_0 + R_0.
$$
Then, an integration by parts yields
$$
\int_\Omega (j(u',A')+A')\cdot \curl B_0=\int_\Omega \mu(u',A')\wedge B_0-\int_{\partial \Omega}(j(u',A')+A')\cdot (B_0\times \nu).
$$
By using the boundary condition $B_0\times \nu=0$ on $\partial\Omega$, we find \eqref{Energy-Splitting}. The inequality for $R_0$ follows directly from the Cauchy-Schwarz inequality. 
\end{proof}
\begin{remark}
Let $\varphi\in C_T^{0,1}(\Omega)$ be a $1$-form. Observe that, by gauge invariance and by integration by parts, we have
$$
\int_\Omega \mu(u,A)\wedge \varphi=\int_\Omega \mu(u',A'+h_{\ex}\curl B_0)\wedge \varphi=\int_\Omega \mu(u',A')\wedge \varphi+h_{\ex}(1-|u|^2)\curl B_0\cdot \curl \varphi.
$$
Moreover, the Cauchy-Schwarz inequality yields
$$
\|\mu(u,A)-\mu(u',A')\|_{C_T^{0,1}(\Omega)^*}\leq C\ve h_{\ex} E_\ve(|u|)^\frac12.
$$
\end{remark}
\subsection{Proof of Theorem \ref{theorem:belowHc1}} 
\begin{proof}
Proposition \ref{prop:energysplitting} yields
\begin{equation}\label{LB}
GL_\ve(u_\ve,A_\ve)\geq h_{\ex}^2J(A_0)+F_\ve(u_\ve',A_\ve')-h_{\ex}\int_\Omega \mu(u_\ve',A_\ve')\wedge B_0+o(\ve^\frac12),
\end{equation}
where $(u_\ve,A_\ve)=(u_0u_\ve',h_\ex A_0+A_\ve')$.

\medskip\noindent
{\bf Step 1. Estimating $F_\ve(u_\ve',A_\ve')$.}
By minimality, we have
\begin{equation}\label{boundJ0}
\inf_{(u,A)\in H^1(\Omega,\C)\times [A_{\ex}+H_{\curl}]}GL_\ve(u,A)\leq GL_\ve (u_0,h_\ex A_0)=h_\ex^2 J(A_0).
\end{equation}
On the other hand, by gauge invariance, we get
\begin{align*}
F_\ve(u_\ve',A_\ve')=F_\ve(u_\ve,A_\ve-h_\ex\curl B_0) &\leq 2F_\ve(u_\ve,A_\ve)+2F_\ve(1,h_\ex \curl B_0)\\ &\leq 2F_\ve(u_\ve,A_\ve)+Ch_\ex^2,
\end{align*}
which combined with \eqref{boundJ0} implies that $F_\ve(u_\ve',A_\ve')\leq M |\log \ve|^2$. We may then apply Theorem \ref{theorem:epslevel} (with $n$ large enough) to obtain
\begin{multline*}
F_\ve(u_\ve',A_\ve')-h_{\ex}\int_\Omega \mu(u_\ve',A_\ve') \wedge B_0 \geq \\ \frac12 |\nu_\ve'|(\Omega)\left( \log \frac1\ve -C \log \log \frac1\ve\right) -h_{\ex}\int_\Omega \nu_\ve' \wedge B_0+o(|\log \ve|^{-2}),
\end{multline*}
where $C>0$ is a universal constant and $\nu_\ve'$ denotes the polyhedral $1$-dimensional current associated to the configuration $(u_\ve',A_\ve')$ by Theorem \ref{theorem:epslevel}.
By noting that 
\begin{equation}\label{nu}
\int_\Omega \nu_\ve' \wedge B_0\leq  |\nu_\ve'|(\Omega)\|B_0\|_*,
\end{equation}
we find
\begin{multline*}
F_\ve(u_\ve',A_\ve')-h_{\ex}\int_\Omega \mu(u_\ve',A_\ve') \wedge B_0 \geq \\ \frac12 |\nu_\ve'|(\Omega)\left(\log \frac1\ve - 2\|B_0\|_*h_\ex -C \log \log \frac1\ve\right)+o(|\log \ve|^{-2}).
\end{multline*}
Writing $h_{\ex}=H_{c_1}^0-K_0\log|\log \ve|$ with $H_{c_1}^0=\dfrac1{2\|B_0\|_*}|\log \ve|$, we get
$$
GL_\ve(u_\ve,A_\ve)\geq h_{\ex}^2J(A_0)+\frac12 |\nu_\ve'|(\Omega)\left(2\|B_0\|_*K_0-C\right)\log \log \frac1\ve  +o(|\log \ve|^{-2}).
$$
By using \eqref{boundJ0}, we deduce that
$$
o(|\log \ve|^{-2})\geq |\nu_\ve'|(\Omega)\left(2\|B_0\|_*K_0-C\right)\log \log \frac1\ve.
$$
Therefore, by letting $K_0\colonequals (2\|B_0\|_*)^{-1}C+1$, we deduce that $|\nu_\ve'|(\Omega)=o(|\log \ve|^{-2})$. In particular, from the vorticity estimate in Theorem \ref{theorem:epslevel} and \eqref{nu}, we deduce that $h_{\ex}\int_\Omega \mu(u_\ve',A_\ve')\wedge B_0=o(|\log \ve|^{-1})$. Therefore, by combining this with \eqref{boundJ0} and \eqref{LB}, we are led to
\begin{equation}\label{Fo(1)}
F_\ve (u_\ve',A_\ve')+\frac12\int_{\R^3\setminus \Omega}|\curl A_\ve'|^2\leq o(|\log \ve|^{-1}).
\end{equation}
In particular, we deduce that $GL_\ve(u_\ve,A_\ve)=h_{\ex}^2J(A_0)+o(|\log\ve|^{-1})$. 

\medskip\noindent
{\bf Step 2. Applying a clearing out result.}
To prove that $(u_\ve,A_\ve)$ is a vortexless configuration we use a clearing out result. Let us define
$$
v_\ve\colonequals e^{-i\varphi_\ve}u_\ve'\quad\mathrm{and}\quad X_\ve\colonequals A_\ve'-\nabla \varphi_\ve,
$$
where $\varphi_\ve$ satisfies
$$
\left\{
\begin{array}{rcll}
\Delta \varphi_\ve&=&\diver A_\ve'&\mathrm{in}\ \Omega\\
\nabla \varphi_\ve\cdot \nu&=&A_\ve'\cdot \nu&\mathrm{on}\ \partial\Omega.
\end{array}\right.
$$
This implies that $X_\ve$ is in the Coulomb gauge, i.e. it satisfies
\begin{equation}\label{condX}
\left\lbrace 
\begin{array}{rcll}
\diver X_\ve&=&0&\mathrm{in}\ \Omega\\
X_\ve \cdot \nu&=&0&\mathrm{on}\ \partial \Omega.
\end{array}
\right.
\end{equation}
Since the configuration $(u_\ve,A_\ve)$ minimizes $GL_\ve$ in $H^1(\Omega,\C)\times [A_{\ex}+H_{\curl}]$, it satisfies the Ginzburg-Landau equations \eqref{GLeq}. By observing that the configurations $(u_\ve,A_\ve)$ and $(v_\ve,X_\ve+h_\ex \curl B_0)$ are gauge equivalent in $\Omega$, we deduce that $v_\ve$ satisfies 
$$
\left\lbrace 
\begin{array}{rcll}
-(\nabla_{X_\ve+h_\ex \curl B_0})^2v_\ve&=&\dfrac1{\ve^2}v_\ve(1-|v_\ve|^2)&\mathrm{in}\ \Omega\\
\nabla_{X_\ve+h_\ex \curl B_0}v_\ve\cdot \nu&=&0&\mathrm{on}\ \partial \Omega.
\end{array}
\right.
$$
Expanding the covariant Laplacian, and using \eqref{condX} and $\curl B_0\cdot \nu =0$ on $\partial\Omega$, which follows from $B_0\times \nu=0$ on $\partial\Omega$, one can rewrite this problem in the form
\begin{equation}\label{condv}
\left\lbrace 
\begin{array}{rcll}
-\Delta v_\ve +i|\log\ve|c(x)\cdot \nabla v_\ve+|\log\ve|^2d(x)v_\ve&=&\dfrac1{\ve^2}v_\ve(1-|v_\ve|^2)&\mathrm{in}\ \Omega\\
\nabla v_\ve\cdot \nu&=&0&\mathrm{on}\ \partial \Omega,
\end{array}
\right.
\end{equation}
where
$$
c(x)\colonequals \frac{2(X_\ve+h_\ex \curl B_0)}{|\log \ve|}\quad \mathrm{and}\quad 
d(x)\colonequals \frac{|X_\ve+h_\ex \curl B_0|^2}{|\log \ve|^2}.
$$
By Remark \ref{remark:B_0} and by standard elliptic regularity theory for solutions of the Ginzburg-Landau equations in the Coulomb gauge, we have
\begin{equation}\label{boundcoef}
\| c\|_{L^\infty(\Omega,\R^3)}, \|\nabla c\|_{L^\infty(\Omega,\R^3)}, \|d\|_{L^\infty(\Omega)}, \|\nabla d\|_{L^\infty(\Omega)}\leq \Lambda_0
\end{equation}
for some constant $\Lambda_0>0$ independent of $\ve$.

In addition, by gauge invariance, we have
$$
F(u_\ve',A_\ve')=F_\ve(v_\ve,X_\ve).
$$
Since $(v_\ve,X_\ve)$ is in the Coulomb gauge, we have
$$
E_\ve(v_\ve)\colonequals F_\ve(v_\ve,0)\leq C F_\ve(v_\ve,X_\ve)
$$
for some universal constant $C>0$. We define $a_\ve(x)=1-d(x)\ve^2|\log \ve|^2$ and observe that
$$
\tilde E_\ve(v_\ve)\colonequals \frac12\int_\Omega |\nabla v_\ve|^2+\frac1{2\ve^2}(a_\ve(x)-|v_\ve|^2)^2\leq E_\ve(v_\ve)+O(\ve|\log \ve|^2).
$$
This combined with \eqref{Fo(1)}, implies that 
\begin{equation}\label{Energyv}
\tilde E_\ve(v_\ve)=o(|\log \ve|^{-1}).
\end{equation}
Finally, from \eqref{condX}, \eqref{condv}, \eqref{boundcoef}, and \eqref{Energyv}, we conclude that all the hypotheses of \cite{Chi}*{Theorem 3} are fulfilled, and therefore
$$
\|1-|u_\ve|\|_{L^\infty(\Omega,\C)}=\|1-|v_\ve|\|_{L^\infty(\Omega,\C)}\to 0 \quad \mathrm{as}\ \ve\to 0.
$$ 
This concludes the proof of the theorem.
\end{proof} 

\section{The first critical field}\label{sec:Hc1}
Let us recall that, given a fixed $\ve>0$, the first critical field is defined as the value $H_{c_1}=H_{c_1}(\ve)$ such that if $h_{\ex}<H_{c_1}$ and $(u_\ve,A_\ve)$ is a minimizer of $GL_\ve$ then $|u_\ve|>0$ in $\Omega$, while if $h_{\ex}>H_{c_1}$ and $(u_\ve,A_\ve)$ minimizes $GL_\ve$ then $u_\ve$ must vanish in $\Omega$. We now prove Theorem \ref{theorem:estimateHc1}.
\begin{proof}
Theorem \ref{theorem:belowHc1} immediately implies that
$$
H_{c_1}^0-K_0\log |\log \ve|\leq H_{c_1}.
$$
It remains to prove that $H_{c_1}\leq H_{c_1}^0+K_1$, for some constant $K_1$ sufficiently large. Let us assume towards a contradiction that $h_{\ex}=H_{c_1}^0+K$ and $(u_\ve,A_\ve)$ minimizes $GL_\ve$ in $H^1(\Omega,\C)\times [A_{\ex}+H_{\curl}]$\footnote{This in particular implies that $(u_\ve,A_\ve)$ satisfies the Ginzburg-Landau equations \eqref{GLeq} and therefore $u_\ve$ is continuous.} and $|u_\ve|>0$. 

\medskip\noindent
{\bf Step 1. Estimating $GL_\ve(u_\ve,A_\ve)$.}
We write $(u_\ve,A_\ve)=(u_0u_\ve',h_{\ex}A_0+A_\ve')$, where $(u_0,h_{\ex}A_0)$ is the approximation of the Meissner solution. Since $|u_\ve'|=|u_\ve|>0$, we deduce that the $1$-dimensional current $\nu_\ve'$ associated to $(u_\ve',A_\ve')$ by Theorem \ref{theorem:epslevel} vanishes identically, and therefore, by taking $n$ large enough, we have
$$
\|\mu(u_\ve',A_\ve')\|_{C_T^{0,1}(\Omega)^*}\leq \frac C{|\log \ve|^2}
$$
The energy-splitting \eqref{Energy-Splitting} then yields
\begin{align*}
GL_\ve(u_\ve,A_\ve)&=h_{\ex}^2J(A_0)+F_\ve(u_\ve',A_\ve')+\frac12\int_{\R^3\setminus \Omega}|\curl A_\ve'|^2+o(|\log \ve|^{-1})\\
&\geq h_{\ex}^2J(A_0)+o(|\log \ve|^{-1}).
\end{align*}
But since $(u_\ve,A_\ve)$ minimizes $GL_\ve$, we have 
$$
GL_\ve(u_\ve,A_\ve)\leq GL_\ve(u_0,h_{\ex}A_0)=h_{\ex}^2J(A_0).
$$
Combining these inequalities, we find
$$
GL_\ve(u_\ve,A_\ve)=h_{\ex}^2J(A_0)+o(|\log \ve|^{-1}).
$$

\medskip\noindent
{\bf Step 2. Definition of a vortex configuration.}
To reach a contradiction, we will show that there exists a configuration $(u^\ve_1,A^\ve_1)$, whose vorticity concentrates along a curve in $X$ for which $\|B_0\|_*$ is almost achieved, such that if $h_{\ex}\geq H_{c_1}^0+K$ then $GL_\ve(u^\ve_1,A^\ve_1)<GL_\ve(u_\ve,A_\ve)$, provided $K\geq K_1$ for some constant $K_1$ independent of $\ve$.

By definition of $\|B_0\|_*$ (recall \eqref{B_0*}), there exists a Lipschitz curve $\vart_\ve\in X$, seen as a $1$-current, with multiplicity $1$ such that $\partial\vart_\ve=0$ relative to $\Omega$ and 
\begin{equation}\label{vort1}
\frac1{|\vart_\ve|(\Omega)}\int_\Omega \vart_\ve \wedge B_0=\|B_0\|_*+o(|\log\ve|^{-1}).
\end{equation}
From the proof of \cite{AlbBalOrl2}*{Theorem 1.1 (ii)}, which in particular uses some results contained in \cite{AlbBalOrl1}, we deduce that there exists $v_\ve \in H^1(\Omega,\C)$ such that
\begin{equation}\label{upper1}
F_\ve(v_\ve,0)\leq \pi |\vart_\ve|(\Omega)|\log\ve|+C_0
\end{equation}
for some constant $C_0>0$ independent of $\ve$, and
$$
\|\mu(v_\ve,0)-2\pi \vart_\ve\|_{C_0^{0,1}(\Omega)^*}=o(|\log\ve|^{-1}),
$$
where $C_0^{0,1}(\Omega)$ denotes the space of Lipschitz vector fields with compact support in $\Omega$. Arguing as in the proof of \cite{JerMonSte}*{Proposition 3.2}, we deduce that
\begin{equation}\label{vort2}
\|\mu(v_\ve,0)-2\pi\vart_\ve\|_{C_T^{0,1}(\Omega)^*}=o(|\log\ve|^{-1}).
\end{equation}

Now we let $(u_1^\ve,A_1^\ve)$ be define by
$$
u_1^\ve=u_0v_\ve, \quad A_1^\ve=h_{\ex}A_0.
$$
Proposition \ref{prop:energysplitting} yields
\begin{equation}\label{expansion}
GL_\ve(u_1^\ve,A_1^\ve)=h_{\ex}^2J(A_0)+F_\ve(v_\ve,0)-h_{\ex}\int_\Omega \mu(v_\ve,0)\wedge B_0+R_0.
\end{equation}
From \eqref{vort1} and \eqref{vort2}, we get
$$
\int_\Omega \mu(v_\ve,0)\wedge B_0=2\pi \|B_0\|_*|\vart_\ve|(\Omega)+o(|\log\ve|^{-1}).
$$
Inserting this and \eqref{upper1} into \eqref{expansion}, we are led to
$$
GL_\ve(u_1^\ve,A_1^\ve)\leq h_{\ex}^2J(A_0)+\pi|\vart_\ve|(\Omega)|\log \ve|+C_0-2\pi\|B_0\|_*h_{\ex}|\vart_\ve|(\Omega)+o(h_{\ex}|\log\ve|^{-1}).
$$

\medskip\noindent
{\bf Step 3. Contradiction.}
Writing $h_{\ex}=H_{c_1}^0+K$ with $H_{c_1}^0=\dfrac1{2\|B_0\|_*}|\log \ve|$, we get
\begin{align*}
GL_\ve(u_1^\ve,A_1^\ve)&\leq h_{\ex}^2J(A_0)+\pi|\vart_\ve|(\Omega)|\log \ve|+C_0-\pi|\vart_\ve|(\Omega)\left(|\log \ve|+2\|B_0\|_*K\right)+o(1)\\
&=h_{\ex}^2J(A_0)+C_0-2\pi\|B_0\|_*K|\vart_\ve|(\Omega)+o(1).
\end{align*}

Clearly $\|B_0\|_*>0$, which implies that there exists a constant $C_1\in (0,1)$ such that $|\vart_\ve|(\Omega)\geq C_1$, provided $\ve$ is small enough. If this is not the case, there exists a sequence of Lipschitz curves $\vart_\ve^n\in X$ such that $|\vart_\ve^n|(\Omega)\leq 1/n$ and which satisfies \eqref{vort1}. If $\vart_\ve^n$ is a loop contained in $\Omega$ then, by the Stokes' theorem, we have
$$
\int_\Omega \vart_\ve^n\wedge B_0\leq \int_{S_n}|\curl B_0|,
$$
where $S_n$ denotes a surface with least area among those whose boundary is $\vart_\ve^n$, i.e. a solution to the associated Plateau's problem. By the isoperimetric inequality, we have
$$
\int_{S_n}|\curl B_0|\leq \|\curl B_0\|_{L^\infty(\Omega,\R^3)} \mathrm{Area}(S_n)\leq C |\vart_\ve^n|(\Omega)^2.
$$
On the other hand, if both different endpoints of $\vart_\ve^n$ belong to $\partial\Omega$, we consider the geodesic connecting the endpoints of $\vart_\ve^n$ on $\partial\Omega$, oriented accordingly to the orientation of $\vart_\ve^n$. We then denote by $\tilde\vart_\ve^n$ the loop formed by the union of $\vart_\ve$ and this geodesic. Since $B_0\times \nu=0$ on $\partial\Omega$, by Stokes' theorem, we have
$$
\int_\Omega \vart^n\wedge B_0\leq\int_{S_n} |\curl B_0|,
$$
where $S_n$ denotes a surface with least area among those whose boundary is $\tilde \vart_\ve^n$. Arguing as above, we conclude that 
$$
\int_{S_n}|\curl B_0|\leq \|\curl B_0\|_{L^\infty(\Omega,\R^3)} \mathrm{Area}(S_n)\leq C \mathrm{Length}(\tilde \vart_\ve^n)^2\leq C(\partial\Omega) |\vart_\ve^n|^2.
$$
Therefore 
$$
\frac1{|\vart_\ve^n|(\Omega)}\int_\Omega \vart_\ve^n\wedge B_0 \leq C|\vart_\ve^n|\leq \frac Cn,
$$
but this is impossible, for $n$ large enough, in view of \eqref{vort1}.

\smallskip
Hence
$$
GL_\ve(u_1^\ve,A_1^\ve)=h_{\ex}^2J(A_0)+C_0-2\pi C_1\|B_0\|_*K+o(1).
$$
By choosing $K_1\colonequals (2\pi C_1\|B_0\|_*)^{-1}C_0+1$, we deduce that, for any $K\geq K_1$ 
$$
GL_\ve(u_1^\ve,A_1^\ve)\leq h_{\ex}^2J(A_0)-1+o(1)<GL_\ve(u_\ve,A_\ve).
$$
This contradicts the fact that $(u_\ve,A_\ve)$ globally minimizes $GL_\ve$. Thus
$$
H_{c_1}\leq H_{c_1}^0+K_1.
$$
\end{proof}
The proof of the upper bound hints that, slightly above the first critical field, vortices will be located close to Lipschitz curves in $X$ for which $\|B_0\|_*$ is (almost) achieved. Let us now study this quantity in a special case.
\begin{proposition}\label{B_0ball}
Consider the special case $\Omega=B(0,R)$ and $H_{0,\ex}=\hat z$ in $B(0,R)$. Then, if $S_1$ denotes the vertical diameter seen as a $1$-current with multiplicity $1$ and oriented in the direction of positive $z$ axis, we have
$$
\|B_0\|_*=\frac1{2R}\int_\Omega S_1\wedge B_0=\frac1{2R}\int_{-R}^R B_0(0,0,z) \cdot \hat z dz=\frac32\left(1-\frac1{\sinh R}\int_0^R\frac{\sinh r}r dr\right).
$$
Moreover, $S_1$ is the only curve in $X$ achieving the maximum in \eqref{B_0*}.
\end{proposition}
\begin{proof}
We follow some ideas from \cite{AlaBroMon}. 

\medskip\noindent
{\bf Step 1. Explicit computation of $B_0$.}
When $\Omega=B(0,R)$ and $H_{0,\ex}=\hat z$ in $B(0,R)$, the solution to \eqref{eqB_0} can be explicitly computed (see \cite{Lon}). By using spherical coordinates $(r,\theta,\phi)$, where $r$ is the Euclidean distance from the origin, $\theta$ is the azimuthal angle, and $\phi$ is the polar angle, we have
$$
B_0=-\frac{3R}{r^2\sinh R}\left(\cosh r-\frac{\sinh r}r\right)\cos\phi \hat r-\frac{3R}{2r^2\sinh R}\left(\cosh r-\frac{1+r^2}r \sinh r\right)\sin \phi \hat \phi-c\hat z,
$$
where $c=\dfrac3{2R\sinh R}\left(\cosh R-\dfrac{1+R^2}R \sinh R\right)$. In particular, we observe that $B_0$ does not depend on the azimuthal angle and therefore it is constant along $\hat \theta$. 

\medskip\noindent
{\bf Step 2. Dimension reduction.}
Let $\Gamma\in X$ with $\int_{B(0,R)} \Gamma\wedge B_0>0$. We will project it along the azimuthal angle onto $B(0,R)^{2\D,+}\colonequals \{(x,z) \in \R^2 \ | \ x^2+z^2<R^2,\ x\geq 0\}$. For this, we consider the map $q:B(0,R)\subset \R^3\to B(0,R)^{2\D,+}$ defined by
$$
q(r,\theta,\phi)=(r\sin\phi,r\cos\phi),
$$
and we let 
$$
\Gamma_{2\D}\colonequals \Gamma\circ q.
$$
It is easy to check that $\partial \Gamma_{2\D}=0$ relative to $B(0,R)^{2\D}$,
$$
\int_{B(0,R)} \Gamma \wedge B_0=\int_{B(0,R)^{2\D,+}} \Gamma_{2\D}\wedge B_0,\quad \mbox{and} \quad |\Gamma_{2\D}|(B(0,R)^{2\D,+})\leq |\Gamma|(B(0,R)).
$$
Therefore
$$
\frac1{|\Gamma|(\Omega)}\int_{B(0,R)}\Gamma\wedge B_0\leq \frac1{|\Gamma_{2\D}|(\Omega)}\int_{B(0,R)^{2\D,+}}\Gamma_{2\D}\wedge B_0.
$$
Even though $\Gamma_{2\D}$ does not necessarily belong to $X$, we can decompose
$$
\Gamma_{2\D}=\sum_{i\in I}\Gamma_i,
$$
where the sum is understood in the sense of currents, $I$ is a finite set of indices, and $\Gamma_i\in X$ for all $i\in I$. In particular,
$$
\int_{B(0,R)^{2\D,+}}\Gamma_{2\D}\wedge B_0\leq \sum_{i\in I} |\Gamma_i|(B(0,R)^{2\D,+})\|B_0\|_* =|\Gamma_{2\D}|(B(0,R)^{2\D,+})\|B_0\|_*.
$$
We deduce that in order to compute $\|B_0\|_*$ it is enough to consider Lipschitz curves $\Gamma \in X$ contained in $B(0,R)^{2\D,+}$ with $\int_{B(0,R)} \Gamma\wedge B_0>0$. From now on we consider $\Gamma$ of this form.

\medskip\noindent
{\bf Step 3. Application of Stokes' theorem.}
If $\Gamma$ has both endpoints on $\partial B(0,R)\cap \partial B(0,R)^{2\D,+}$, we then define $\tilde \Gamma$ as the loop formed by the union of $\Gamma$ and the curve lying on $\partial B(0,R)\cap \partial B(0,R)^{2\D,+}$ which connects the end points of $\Gamma$ oriented accordingly to the orientation of $\Gamma$. Since $B_0\times \nu =0$ on $\partial B(0,R)$, the Stokes' theorem yields
\begin{equation}\label{stokes1}
\int_{B(0,R)^{2\D,+}}\Gamma\wedge B_0=\int_{B(0,R)^{2\D,+}}\tilde \Gamma\wedge B_0=\int_{S_\Gamma} \curl B_0\cdot \hat y,
\end{equation}
where $S_\Gamma$ is the surface enclosed by $\tilde \Gamma$. Of course if $\Gamma$ is a loop contained in $B(0,R)^{2\D,+}$ then the Stokes' theorem gives
$$
\int_{B(0,R)^{2\D,+}}\Gamma\wedge B_0=\int_{S_\Gamma} \curl B_0\cdot \hat y,
$$ 
where $S_\Gamma$ is the surface enclosed by $\Gamma$. 

An explicit computation gives
\begin{equation}\label{curl}
\curl B_0\cdot \hat y=\frac{3R}{2\sinh R}\left(\cosh r-\frac{\sinh r}r \right)\frac{\sin \phi}r\geq 0\quad \mathrm{in}\ B(0,R)^{2\D,+}.
\end{equation}
In what follows we use the notation
$$
f(r)\colonequals \frac{3R}{2\sinh R}\left(\cosh r-\frac{\sinh r}r \right).
$$

\medskip\noindent
{\bf Step 4. Estimate for curves with endpoints on  $\partial B(0,R)\cap \partial B(0,R)^{2\D,+}$.} For $a,b\in[0,\pi]$ with $a<\pi-b$ let us define
$$
S_{a,b}\colonequals \{(r,\phi)\ | \ 0\leq r\leq R,\ a\leq \phi \leq \pi-b \}.
$$ 
We let $\phi_1,\phi_2$ be the maximum angles for which $S_\Gamma \subset S_{\phi_1,\phi_2}$. From \eqref{stokes1} and \eqref{curl}, we deduce that 
\begin{align*}
\int_{B(0,R)^{2\D,+}}\Gamma\wedge B_0 \leq \int_{S_{\phi_1,\phi_2}}\curl B_0\cdot \hat y &=\int_0^R \int_{\phi_1}^{\pi-\phi_2}f(r)\sin\phi d\phi dr\\
&=(\cos \phi_1 +\cos \phi_2)\int_0^R f(r)dr
\end{align*}
On the other hand, by definition of $\phi_1,\phi_2$, $S_\Gamma$ intersects the rays $\{(r,\phi_1) \ | \ 0\leq r\leq R \}$ and $\{(r,\phi_2) \ | \ 0\leq r\leq R \}$.
Since the endpoints of $\Gamma$ belong to $\partial B(0,R)\cap \partial B(0,R)^{2\D,+}$, a simple geometric argument shows that 
$$
|\Gamma|(B(0,R)^{2\D,+})\geq d((R,\phi_1),(R,\phi_2)).
$$
The law of cosines yields $d((R,\phi_1),(R,\phi_2))=R\sqrt{2(1-\cos(\pi-\phi_1-\phi_2))}$. Hence
$$
\frac1{|\Gamma|(B(0,R)^{2\D,+})}\int_{B(0,R)^{2\D,+}}\Gamma\wedge B_0\leq \frac{\cos \phi_1+\cos \phi_2}{\sqrt{2(1-\cos(\pi-\phi_1-\phi_2))}}\frac{\int_0^R f(r)dr}R.
$$
We now estimate the right-hand side of this inequality. Let us observe that
$$
\cos\phi_1+\cos\phi_2=2\cos\left(\frac{\phi_1+\phi_2}2\right) \cos\left(\frac{\phi_1-\phi_2}2\right)
$$
and
$$
\cos(\pi-\phi_1-\phi_2)=
\cos(\phi_1+\phi_2)=\cos^2\left(\frac{\phi_1+\phi_2}2\right)-\sin^2\left(\frac{\phi_1+\phi_2}2\right)=2\cos^2\left(\frac{\phi_1+\phi_2}2\right)-1.
$$
Using $0\leq\frac{\phi_1+\phi_2}2<\frac\pi2$, we deduce that
$$
\frac{\cos\phi_1+\cos\phi_2}{\sqrt{2(1-\cos(\pi-\phi_1-\phi_2))}}= \cos\left(\frac{\phi_1-\phi_2}2\right)\leq 1,
$$
with equality if and only if $\phi_1=\phi_2$. Therefore
\begin{align*}
\frac1{|\Gamma|(B(0,R)^{2\D,+})}\int_{B(0,R)^{2\D,+}}\Gamma\wedge B_0\leq\frac{\int_0^R f(r)dr}R&=\frac32\left(1-\frac1{\sinh R}\int_0^R\frac{\sinh r}r dr\right)\\
&=\frac1{2R}\int_{B(0,R)} S_1\wedge B_0.
\end{align*}
Besides, from the previous computations we easily deduce that the inequality is strict if $\Gamma \neq S_1$.

\medskip\noindent
{\bf Step 5. Estimate for loops in $B(0,R)^{2\D,+}$.}
Let us define $0<r_0<R$ as the minimum radius such that 
$$
S_\Gamma \subset B(0,r_0)^{2\D,+}.
$$
In particular, $S_\Gamma\cap (\partial B(0,r_0)\cap \partial B(0,r_0)^{2\D,+})\neq \emptyset$. We can then use the estimate provided in the previous step and conclude that
$$
\frac1{|\Gamma|(B(0,R)^{2\D,+})}\int_{B(0,R)^{2\D,+}}\Gamma\wedge B_0\leq \frac32\left(1-\frac1{\sinh r_0}\int_0^{r_0}\frac{\sinh r}r dr\right).
$$
One can check that the function $R\to \frac1{\sinh R}\int_0^R\frac{\sinh r}r dr$ is strictly decreasing in $[0,\infty)$ and therefore 
$$
\frac1{|\Gamma|(B(0,R)^{2\D,+})}\int_{B(0,R)^{2\D,+}}\Gamma\wedge B_0< \frac32\left(1-\frac1{\sinh R}\int_0^{R}\frac{\sinh r}r dr\right).
$$
This concludes the proof of the proposition.
\end{proof}  

\section{A Meissner-type solution beyond the first critical field}\label{sec:Meissner}
In this section, we present the proof of Theorem \ref{theorem:meissner}.
\begin{proof} 
{\bf Step 1. Existence of a locally minimizing vortexless configuration.}
Let us introduce the set
$$
U=\left\{(u,A)\in H^1(\Omega,\C)\times [A_{\ex}+H_{\curl}] \ | \ F_\ve(u',A')< \ve^\frac23 \right\},
$$
where $u'=u_0^{-1}u$ and $A'=A-h_{\ex}A_0$.
Consider a minimizing sequence $\{(\tilde u_n,\tilde A_n)\}_n\in U$. Lemma \ref{lemma:decompositionR3} yields a gauge transformed sequence $\{u_n,A_n\}_n\in [A_{\ex}+\dot H^1_{\diver=0}]$, that in particular satisfies $F_\ve(u_n',A_n')=F_\ve(\tilde u_n',\tilde A_n')<\ve^\frac23$. In addition, by gauge invariance, we can take this sequence to satisfy $A_n'\cdot \nu=0$ on $\partial\Omega$. Then arguing as in Proposition \ref{prop:minimization}, we deduce that (up to subsequence) $\{(u_n,A_n-A_{\ex})\}_n$ converges to some $(u,A-A_{\ex})$ weakly in $H^1(\Omega,\C)\times \dot H^1_{\diver=0}$. Arguing again as in Proposition \ref{prop:minimization}, we find
$$
F_\ve(u',A')\leq \liminf_n F_\ve(u_n',A_n') \quad \mathrm{and} \quad GL_\ve(u,A)\leq \liminf_n GL_\ve(u_n,A_n).
$$
Hence $(u,A)\in \overline U\cap H^1(\Omega,\C)\times [A_{\ex}+ \dot H_{\diver=0}^1]$ satisfies $A'\cdot \nu=0$ on $\partial\Omega$ and minimizes $GL_\ve$ over $\overline U$.

Let us now prove that $(u,A)\in U$. We consider, for $\de=\de(\ve)=c_1\ve^\frac13$ and $\ve$ sufficiently small, the grid $\GG(b_\ve,R_0,\de)$ associated to $(u',A')$ by \cite{Rom}*{Lemma 2.1}\footnote{Though it is not explicitly written in the paper, if $F_\ve(u',A')\leq \ve^\gamma$ with $\gamma\in (0,1)$ then \cite{Rom}*{Lemma 2.1} holds for any $c_1\ve^\frac{1+\gamma}4\leq \de \leq \de_0$. The proof of this is exactly the same.}. In particular, using the same notation as in this lemma, we have
\begin{align*}
|u_\ve|>5/8\quad \mathrm{on}\ &\RR_1(\GG(b_\ve,R_0,\delta)),\\
I_\ve^1\colonequals\int\limits_{\RR_1(\GG(b_\ve,R_0,\delta))}e_\ve(u',A')d\H^1&\leq  C \de^{-2}F_\ve(u',A')\leq C\ve^{-\frac23}\ve^{\frac23},\\
I_\ve^2\colonequals\int\limits_{\RR_2(\GG(b_\ve,R_0,\delta))}e_\ve(u',A')d\H^2&\leq  C \de^{-1}F_\ve(u',A')\leq C\ve^{-\frac13}\ve^{\frac23}.
\end{align*}
Then, the 2D ball construction method (see \cite{SanSerBook}*{Theorem 4.1}) applied in the 2D skeleton $\RR_2(\GG(b_\ve,R_0,\de))$ of the grid implies that for each face $\omega$ of a cube of the grid, the collection of connected components $S_{i,\omega}$ of $\{x\in\omega \ | \ |u(x)|<1/2\}$ whose degree is different from zero is empty (see \cite{Rom}*{Section 4}). We thus deduce that the $1$-current $\nu_\ve'$, which approximates well the vorticity $\mu(u',A')$, vanishes identically in $\overline\Omega$. Then, from the proof of Theorem \ref{theorem:epslevel} (see \cite{Rom}*{Section 8}), we find
$$
\|\mu(u',A')\|_{C_T^{0,1}(\Omega)^*}\leq C\de F_\ve(u',A')+C\ve(1+I_\ve^1+I_\ve^2)\leq C(\de +\ve \de^{-2})F_\ve(u',A').
$$
Let us now use Proposition \ref{prop:energysplitting}. From the previous inequality and since $\alpha<\frac13$, we have
\begin{equation}\label{condalpha}
\left| h_{\ex}\int_\Omega \mu(u',A')\wedge B_0\right|\leq Ch_{\ex}(\de +\ve \de^{-2})F_\ve(u',A')\leq C\ve^{\frac13-\alpha}F_\ve(u',A')=o(\ve^\frac23).
\end{equation}
On the other hand
$$
R_0\leq C\ve h_{\ex}^2E_\ve(|u'|)^\frac12\leq C\ve h_{\ex}^2F_\ve(u',A')^\frac12\leq C\ve^{1-2\alpha}\ve^\frac13=o(\ve^\frac23).
$$
The energy-splitting \eqref{Energy-Splitting} then yields 
$$
GL_\ve(u,A)=h_{\ex}^2J(A_0)+F_\ve(u',A')+\frac12\int_{\R^3\setminus \Omega} |\curl A'|^2+o(\ve^\frac23).
$$
But, since $(u_0,h_{\ex}A_0)$ belongs to $U$, we have
$$
GL_\ve(u,A)\leq GL_\ve(u_0,h_{\ex}A_0)=h_{\ex}^2J(A_0).
$$
We thus deduce that
\begin{equation}\label{energyu'A'}
F_\ve(u',A')+\frac12 \int_{\R^3\setminus\Omega}|\curl A'|^2=o(\ve^\frac23),
\end{equation}
and therefore $(u,A)\in U$ provided $\ve$ is small enough. 

\smallskip
Now, since $U$ is open in $H^1(\Omega,\C)\times [A_{\ex}+H_{\curl}]$, the minimizer $(u,A)$ must be a critical point of $GL_\ve$ and therefore satisfies the Ginzburg-Landau equations \eqref{GLeq}. Arguing as in the proof of Theorem \ref{theorem:belowHc1}, we deduce that $(u,A)$ is a vortexless configuration and
$$
\|1-|u|\|_{L^\infty(\Omega,\C)}=\|1-|u'|\|_{L^\infty(\Omega,\C)}=o(1)\quad \mathrm{as}\ \ve\to 0.
$$
We note that we have omitted in our notation the dependence on $\ve$ of the minimizer $(u,A)$.

\medskip\noindent
{\bf Step 2. Characterization of $(u',A')$.}
Since $A'\cdot \nu =0$ on $\partial\Omega$, we have
\begin{equation}\label{estimateu'A'}
\int_\Omega |\nabla u'|^2\leq \int_\Omega |\nabla_{A'} u'|^2+|A'|^2|u'|^2\leq CF_\ve(u',A')\leq C\ve^\frac23
\end{equation}
for some universal constant $C>0$.

Observe that, using the Poincar\'e-Wirtinger inequality, we have
\begin{equation}\label{meanvalue}
\int_\Omega |u'-\underline u'|^2\leq C\int_\Omega |\nabla u'|^2,\quad \mathrm{where}\ \underline u'=\frac1{|\Omega|}\int_\Omega u'.
\end{equation}
In addition, we have
$$
\int_\Omega \big| |u'|-|\underline u'|\big|^2\leq \int_\Omega |u'-\underline u'|^2
$$
and 
$$
\int_\Omega (1-|u'|)^2\leq \int_\Omega (1-|u'|^2)^2\leq 4\ve^2 F_\ve(u',A')\leq 4\ve^{2+\frac23}.
$$
We deduce that $\|1-|\underline u'|\|_{L^2(\Omega,\C)}\leq C\ve^{1+\frac13}$. But $\underline u'$ is a constant, thus $\underline u'=e^{i\theta_\ve}+O(\ve^{1+\frac13})$ for some $\theta_\ve\in[0,2\pi]$. By combining with \eqref{meanvalue} and \eqref{estimateu'A'}, we find
\begin{equation}\label{u'L^2}
\int_\Omega |u'-e^{i\theta_\ve}|^2\leq C\ve^\frac23.
\end{equation}
Thus
\begin{equation}\label{estimateu'}
\inf_{\theta\in[0,2\pi]} \|u'-e^{i\theta}\|_{H^1(\Omega,\C)}\to 0\quad \mathrm{as}\ \ve\to 0.
\end{equation}
On the other hand, from \eqref{energyu'A'}, we have $\|\curl A'\|^2_{L^2(\R^3,\R^3)}=o(\ve^\frac23)$, which combined with the fact that $A'=A-h_{\ex}A_0\in \dot H^1_{\diver=0}$ implies that $\|A'\|_{\dot H^1_{\diver=0}}^2=o(\ve^\frac23)$.

In particular, by noting that $(e^{i\theta},h_{\ex}\curl B_0)$ is gauge equivalent to $(1,h_{\ex} \curl B_0)$ in $\Omega$ for any $\theta\in[0,2\pi]$, we deduce that (up a gauge transformation) the configuration $(u',A'+h_{\ex}\curl B_0)$, which is gauge equivalent to $(u,A)$ in $\Omega$, gets closer and closer in the $H^1(\Omega,\C)\times H^1(\Omega,\R^3)$-norm to $(1,h_{\ex}\curl B_0)$.

\medskip\noindent
{\bf Step 3. $(u,A)$ globally approaches $(u_0,h_{\ex}A_0)$.}
Observe that, for any $\theta \in [0,2\pi]$, we have
$$
\int_\Omega |u-e^{i\theta}u_0|^2=\int_\Omega |u'u_0-e^{i\theta}u_0|^2=\int_\Omega |u'-e^{i\theta}|^2
$$
and
$$
\int_\Omega |\nabla (u-e^{i\theta}u_0)|^2\leq \int_\Omega |\nabla u_0|^2|u'-e^{i\theta}|^2+\int_\Omega |\nabla u'|^2.
$$
From \eqref{estimateu'}, we deduce that 
$$
\inf_{\theta\in[0,2\pi]} \|u-e^{i\theta}u_0\|_{L^2(\Omega,\C)}\to 0\quad \mathrm{as}\ \ve \to 0.
$$
Recall that $u_0=e^{ih_{\ex}\phi_0}$ and that $A_0$ satisfies the Euler-Lagrange equation \eqref{EulerLagrangeA0}. Since $\curl(H_0-H_{0,\ex})=\curlcurl (A_0-A_{0,\ex})=-\Delta(A_0-A_{0,\ex})$, standard elliptic regularity theory implies that $\phi_0=A_0-\curl B_0\in L^\infty(\Omega)$. Therefore
$$
\int_\Omega |\nabla u_0|^2|u'-e^{i\theta}|^2\leq h_{\ex}^2\|\nabla \phi_0\|_{L^\infty(\Omega)}^2 \|u'-e^{i\theta}\|_{L^2(\Omega,\C)}^2.
$$
This combined with \eqref{u'L^2} for $\theta=\theta_\ve$, yields
$$
\int_\Omega |\nabla u_0|^2|u'-e^{i\theta_\ve}|^2\leq C\ve^{-2\alpha}\ve^\frac23.
$$
Since $\alpha<\frac13$, the right-hand side converges to $0$ as $\ve\to0$. Using once again \eqref{estimateu'}, we obtain
$$
\inf_{\theta\in[0,2\pi]}\int_\Omega |\nabla (u-e^{i\theta}u_0)|^2\to 0\quad \mathrm{as}\ \ve\to 0.
$$
Hence
$$
\inf_{\theta\in[0,2\pi]} \|u-e^{i\theta}u_0\|_{H^1(\Omega,\C)}\to 0\quad \mathrm{as}\ \ve\to 0.
$$
Moreover, we have
$$
\|A-h_{\ex}A_0\|_{\dot H^1_{\diver=0}}=\|A'\|_{\dot H^1_{\diver=0}}\to 0\quad \mathrm{as}\ \ve \to 0.
$$

\medskip
We have hence shown that, up to a gauge transformation in $\R^3$, the solution $(u,A)$ gets closer and closer in the $H^1(\Omega,\C)\times \dot H^1_{\diver=0}$-norm to $(u_0,h_{\ex}A_0)$. In addition, up to a (different) gauge transformation in $\Omega$, the solution approaches in the $H^1(\Omega,\C)\times H^1(\Omega,\R^3)$-norm the configuration $(1,h_{\ex}\curl B_0)$.
\end{proof}

\begin{remark}\label{remark:condalpha}
The assumption $h_{\ex}\leq \ve^{-\alpha}$ for $\alpha<\frac13$ is needed to prove that 
$$
\left| h_{\ex}\int_\Omega \mu(u',A')\wedge B_0\right|\leq o(F_\ve(u',A'));
$$
see \eqref{condalpha}. If $\alpha\geq \frac13$, we are not able to show this, and our strategy to prove that $(u,A)\in U$ then fails.
\end{remark}

\section{Uniqueness of locally minimizing vortexless configurations}\label{sec:Uniqueness}
In this section we prove Theorem \ref{theorem:uniqueness}. We follow the same strategy as in \cite{Ser2}*{Section 2}.
\begin{proof}
Since any pair $(\tilde u,\tilde A)\in H^1(\Omega,\C)\times [A_{\ex}+ H_{\curl}]$ is gauge equivalent to $(u,A)\in H^1(\Omega,\C)\times H^1_{\diver=0}$ with $A_j'\cdot \nu=0$ on $\partial\Omega$, it is enough to prove the theorem in this gauge.

Let us assume towards a contradiction that there are two distinct locally minimizing vortexless solutions $(u_j,A_j)=(u_0u_j',h_{\ex}A_0+A_j')$ to \eqref{GLeq} with $u_j\in H^1(\Omega,\C)$, $|u_j|\geq c$ for some $c\in(0,1)$, $A_j'\in H^1_{\diver=0}$, $A_j'\cdot \nu=0$ on $\partial\Omega$, and
$$
F_\ve(u_j',A_j')\leq C\ve^{1+\de} \quad \mathrm{for}\ j=1,2,
$$
for some $\de>0$. As we shall see, this estimate is crucial to prove the theorem.

Since $|u_j'|=|u_j|\geq c>0$, we can write $u_j'=\eta_je^{i\phi_j}$ in $\Omega$ for $j=1,2$. Note that the functions $\phi_0,\phi_1,\phi_2\in H^2(\Omega)$ can be extended to functions in $H^2(\R^3)$. Therefore, for $j=1,2$, $(u_j,A_j)$ is gauge equivalent to $(\eta_j,\tilde A_j)$ with
$$
\tilde A_j=h_{\ex}(A_0-\nabla \phi_0)+A_j'-\nabla \phi_j.
$$

\medskip\noindent
{\bf Step 1. Estimating $\|\tilde A_j\|_{L^\infty(\Omega,\R^3)}$.}
Let us show that, for $j=1,2$, we have
\begin{equation}\label{estimateA_j}
\|\tilde A_j\|_{L^\infty(\Omega,\R^3)}\leq o(\ve^{-1}).
\end{equation}
By gauge equivalence, $(u_j',h_{\ex}(A_0-\nabla \phi_0)+A_j')$ solves \eqref{GLeq}. Then, by standard elliptic regularity theory for solutions of the Ginzburg-Landau equations in the Coulomb gauge, we have
$$
\|h_{\ex}A_0\curl B_0+A_j'\|_{L^\infty(\Omega,\R^3)}\leq Ch_{\ex}\quad \mathrm{and}\quad \|\nabla u_j'\|_{L^\infty(\Omega,\R^3)}\leq C\ve^{-1}.
$$
Since
\begin{equation}\label{estimategrad}
\|\nabla \eta_j\|_{L^\infty(\Omega,\R^3)}+\|\nabla \phi_j\|_{L^\infty(\Omega,\R^3)}\leq 2\|\nabla u_j'\|_{L^\infty(\Omega,\R^3)},
\end{equation}
and $h_{\ex}=o(\ve^{-1})$, we find
\begin{equation}\label{firstestimate}
\|\tilde A_j\|_{L^\infty(\Omega,\R^3)}\leq \|h_{\ex}\curl B_0+A_j'\|_{L^\infty(\Omega,\R^3)}+\|\nabla \phi_j\|_{L^\infty(\Omega,\R^3)}\leq C\ve^{-1}.
\end{equation}
We will now improve this estimate. By gauge equivalence, $(\eta_j,\tilde A_j)$ solves \eqref{GLeq}. In particular, the second Ginzburg-Landau equation in $\Omega$ reads
$$
\curl^2(\tilde A_j - A_{\ex})=-\eta_j^2\tilde A_j.
$$
This implies that $\diver(\eta_j^2\tilde A_j)=0$ in $\Omega$. In addition, the boundary condition $\nabla_{\tilde A_j}\eta_j\cdot \nu=0$ on $\partial \Omega$, implies, in particular, that $\nabla \phi_j \cdot \nu =0$ on $\partial\Omega$. Therefore, $\phi_j$ satisfies the elliptic problem
$$
\left\{ 
\begin{array}{rcll}
\Delta \phi_j&=&\dfrac2{\eta_j} \nabla \eta_j \cdot \tilde A_j&\mathrm{in}\ \Omega\\
\nabla \phi_j\cdot \nu&=&0&\mathrm{on}\ \partial\Omega. 
\end{array}
\right.
$$
Because $\eta_j\geq c>0$, we deduce that, for any $p>1$,
\begin{equation}\label{Laplacian}
\|\Delta \phi_j \|_{L^p(\Omega)}\leq C\|\tilde A_j\|_{L^\infty(\Omega,\R^3)}\|\nabla \eta_j\|_{L^p(\Omega,\R^3)}\leq C\ve^{-1}\|\nabla \eta_j\|_{L^p(\Omega,\R^3)},
\end{equation}
where the last inequality is obtained by using \eqref{firstestimate}.

On the other hand, since $A_j'\cdot\nu =0$ on $\partial \Omega$, we have 
$$
\|\nabla u_j'\|^2_{L^2(\Omega,\C^3)},\ \|A_j'\|^2_{L^2(\Omega,\R^3)}\leq C F_\ve(u_j',A_j')\leq C\ve^{1+\de}. 
$$
This implies that
$$
\int_\Omega |\nabla \eta_j|^2+\eta_j^2|\nabla \phi_j|^2=\int_\Omega |\nabla u_j'|^2\leq C\ve^{1+\de}.
$$
In addition, by interpolation, for any $p>1$, we have
$$
\|\nabla \eta_j\|_{L^p(\Omega,\R^3)}\leq C \|\nabla \eta_j\|_{L^\infty(\Omega,\R^3)}^{1-\frac2p}\|\nabla \eta_j\|_{L^2(\Omega,\R^3)}^{\frac2p}.
$$
Combining the previous two inequalities with \eqref{estimategrad}, yields
$$
\|\nabla \eta_j\|_{L^p(\Omega,\R^3)}\leq C\ve^{-1+\frac2p}\ve^\frac{1+\de}p=C\ve^\frac{3+\de-p}p.
$$
Combining with \eqref{Laplacian} for $p=3+\frac\de2>3$, we find
$$
\|\Delta \phi_j \|_{L^p(\Omega)}\leq C\ve^{-1}\ve^{\frac\de{6+\de}}=o(\ve^{-1}).
$$
By an elliptic estimate and Sobolev embedding, we then obtain
$$
\|\nabla \phi_j \|_{L^\infty(\Omega)}\leq o(\ve^{-1}).
$$
Thus
$$
\|\tilde A_j\|_{L^\infty(\Omega,\R^3)}\leq \|h_{\ex}\curl B_0+A_j'\|_{L^\infty(\Omega,\R^3)}+\|\nabla \phi_j\|_{L^\infty(\Omega,\R^3)}\leq Ch_{\ex} + o(\ve^{-1})=o(\ve^{-1}).
$$

\medskip\noindent
{\bf Step 2. Energy estimate.}
Let us prove that 
$$
Y\colonequals \frac{GL_\ve(\eta_1,\tilde A_1)+GL_\ve(\eta_2,\tilde A_2)}2-GL_\ve\left(\frac{\eta_1+\eta_2}2,\frac{\tilde A_1+\tilde A_2}2\right)>0
$$
First, observe that
$$
\int_\Omega |\nabla_{\tilde A_j} \eta_j|^2=\int_\Omega |\nabla \eta_j|^2+\eta_j^2|\tilde A_j|^2.
$$
We write $Y=Y_0+Y_1+Y_2+Y_3$ with
\begin{align*}
Y_0&=\frac12\int_\Omega |\nabla \eta_1|^2+|\nabla \eta_2|^2-\int_\Omega \left|\nabla \left(\frac{\eta_1+\eta_2}2\right)\right|^2,\\
Y_1&=\frac12\int_\Omega \eta_1^2|\tilde A_1|^2+\eta_2^2|\tilde A_2|^2-\int_\Omega\left(\frac{\eta_1+\eta_2}2\right)^2\left|\frac{\tilde A_1+\tilde A_2}2\right|^2,\\
Y_2&=\frac12\left(\frac1{4\ve^2}\int_\Omega (1-\eta_1^2)^2+(1-\eta_2^2)^2\right)-\frac1{4\ve^2}\int_\Omega \left(1-\left(\frac{\eta_1+\eta_2}2\right)^2\right)^2,\\
Y_3&=\frac12\int_\Omega |\curl \tilde A_1-H_{\ex}|^2+|\curl \tilde A_2-H_{\ex}|^2-\int_\Omega \left|\curl\left(\frac{\tilde A_1+\tilde A_2}2 \right)-H_{\ex} \right|^2.
\end{align*}
Note that, by convexity we have $Y_0,Y_3\geq 0$.

On the other hand, arguing exactly as in the proof of \cite{Ser2}*{Lemma 2.5}, we get
\begin{multline*}
Y_1=\frac1{16}\int_\Omega |\eta_1-\eta_2|^2|\tilde A_1+\tilde A_2|^2+4\eta_1^2|\tilde A_1-\tilde A_2|^2\\
-(\eta_1-\eta_2)(\tilde A_1-\tilde A_2)\cdot\left(\tilde A_1(2\eta_1+4\eta_2)+\tilde A_2(6\eta_1+8\eta_2)\right)
\end{multline*}
and
$$
Y_2\geq \frac3{64\ve^2}\int_\Omega (\eta_1-\eta_2)^2.
$$
Let us prove that $Y_1+Y_2>0$. We consider three cases.
\begin{itemize}[leftmargin=*]
\item If $\eta_1=\eta_2$ then 
$$
Y_1+Y_2\geq \int_\Omega 4\eta_1^2|\tilde A_1-\tilde A_2|^2>0.
$$
\item If $\tilde A_1=\tilde A_2$ then $Y_1\geq 0$. Therefore
$$
Y_1+Y_2\geq Y_2\geq\frac3{64\ve^2}\int_\Omega (\eta_1-\eta_2)^2>0.
$$
\item If $\eta_1\neq \eta_2$ and $\tilde A_1\neq \tilde A_2$ then
$$
Y_1\geq \frac1{16}\int_\Omega |\eta_1-\eta_2|^2|\tilde A_1+\tilde A_2|^2+4\eta_1^2\left|\tilde A_1-\tilde A_2\right|^2-|\eta_1-\eta_2||\tilde A_1-\tilde A_2|(6|\tilde A_1|+14|\tilde A_2|).
$$
By the Cauchy-Schwarz inequality, we have
\begin{multline*}
\int_\Omega|\eta_1-\eta_2||\tilde A_1-\tilde A_2|(6|\tilde A_1|+14|\tilde A_2|)\\
\leq 14(\|\tilde A_1\|_{L^\infty(\Omega,\R^3)}+\|\tilde A_2\|_{L^\infty(\Omega,\R^3)})\|\eta_1-\eta_2\|_{L^2(\Omega)}\|\tilde A_1-\tilde A_2\|_{L^2(\Omega,\R^3)},
\end{multline*}
which combined with \eqref{estimateA_j}, yields
$$
\int_\Omega|\eta_1-\eta_2||\tilde A_1-\tilde A_2|(6|\tilde A_1|+14|\tilde A_2|)\leq o(\ve^{-1})\|\eta_1-\eta_2\|_{L^2(\Omega)}\|\tilde A_1-\tilde A_2\|_{L^2(\Omega,\R^3)}.
$$
On the other hand,
$$
\int_\Omega \frac14\eta_1^2|\tilde A_1-\tilde A_2|^2+\frac3{64\ve^2}(\eta_1-\eta_2)^2 \geq \frac{9}{32\ve}\|\eta_1-\eta_2\|_{L^2(\Omega)}\|\tilde A_1-\tilde A_2\|_{L^2(\Omega,\R^3)}.
$$
Hence, if $\ve$ is small enough then $Y_1+Y_2>0$.
\end{itemize}
We have thus proved that $Y>0$.

\medskip\noindent
{\bf Step 3. Contradiction.}
Assume without loss of generality that 
$$
GL_\ve(\eta_1,\tilde A_1)\leq GL_\ve(\eta_2,\tilde A_2).
$$
From the previous step, we have
$$
GL_\ve\left(\frac{\eta_1+\eta_2}2,\frac{\tilde A_1+\tilde A_2}2\right)<\frac{GL_\ve(\eta_1,\tilde A_1)+GL_\ve(\eta_2,\tilde A_2)}2\leq GL_\ve(\eta_2,\tilde A_2).
$$
A standard argument then shows that, for any $t\in(0,1)$, 
$$
GL_\ve\left(t\eta_1+(1-t)\eta_2,t\tilde A_1+(1-t)\tilde A_2\right)<GL_\ve(\eta_2,\tilde A_2),
$$
contradicting the fact that $(\eta_2,\tilde A_2)$ is a local minimizer of the energy. Hence $(\eta_1,\tilde A_1)=(\eta_2,\tilde A_2)$. This concludes the proof.
\end{proof}

\appendix
\section{Improved estimates for locally minimizing vortexless configurations}\label{sec:appendix}
\begin{proposition}\label{prop:vort}
Let $(u,A)\in H^1(\Omega,\C)\times H^1(\Omega,\R^3)$ with $u$ continuous and $|u|\geq c$ for some $c\in(0,1)$. Then
$$
\|\mu(u,A)\|_{C_T^{0,1}(\Omega,\R^3)^*}\leq C\ve F_\ve(u,A).
$$
\end{proposition}
\begin{proof}
Let $\varphi \in C_T^{0,1}(\Omega,\R^3)$. By integration by parts, we have
$$
\int_\Omega \mu(u,A)\wedge \varphi =\int_\Omega (j(u,A)+A)\cdot \curl \varphi. 
$$
Since $|u|\geq c>0$, we can write $u=|u|e^{i\phi}$. A straightforward computation, shows that
$$
j(u,A)+A=|u|^2\nabla \phi +(1-|u|^2)A=(1-|u|^2)(A-\nabla \phi)+\nabla \phi.
$$
Observe that, by integration by parts, we have $\int_\Omega \nabla \phi \cdot \curl \varphi=0$. Then, from the Cauchy-Schwarz inequality, we deduce that
$$
\left|\int_\Omega (j(u,A)+A)\cdot \curl \varphi\right|\leq \int_\Omega (1-|u|^2)|A-\nabla \phi||\curl \varphi|\leq C\|\curl \varphi\|_{L^\infty(\Omega,\R^3)}\ve F_\ve(u,A).
$$
Hence
$$
\|\mu(u,A)\|_{C_T^{0,1}(\Omega,\R^3)^*}\leq C\ve F_\ve(u,A).
$$
\end{proof}
With this estimate at hand, we prove the following result.
\begin{proposition}\label{prop:estimteFree} 
Denote $(u_0,h_{\ex}A_0)$ the approximation of the Meissner solution. Let $(u,A)=(u_0u',h_{\ex}A_0+A')\in H^1(\Omega,\C)\times [A_{\ex}+H_{\curl}]$ with $u$ continuous and $|u|\geq c$ for some $c\in(0,1)$. If $h_{\ex}\leq \ve^{-\alpha}$ for some $\alpha\in \left(0,\frac14\right)$ and $GL_\ve(u,A)\leq GL_\ve(u_0,h_{\ex}A_0)$ then, for any $\ve$ sufficiently small, we have
$$
F_\ve(u',A')+\frac12 \int_{\R^3\setminus\Omega} |\curl A'|^2\leq C\ve^{1+\de}
$$
for some $\de\in(0,1)$.
\end{proposition}
\begin{proof}
Let us first observe that, since $GL_\ve(u,A)\leq GL_\ve(u_0,h_{\ex}A_0)= h_{\ex}^2J(A_0)$, we have
\begin{equation}\label{first}
F_\ve(u',A')\leq Ch_{\ex}^2 \leq C \ve^{-2\alpha}
\end{equation}
for some constant $C>0$. We will now use Proposition \eqref{prop:energysplitting} to improve this estimate. By combining \eqref{Energy-Splitting} with the $GL_\ve(u,A)\leq GL_\ve(u_0,h_{\ex}A_0)$, we find
$$
F_\ve(u',A')+\frac12 \int_{\R^3\setminus\Omega} |\curl A'|^2\leq h_{\ex} \int_\Omega \mu(u',A')\wedge B_0+C\ve h_{\ex}^2E_\ve(|u'|)^\frac12.
$$
From Proposition \ref{prop:vort} and $E_\ve(|u'|)\leq F_\ve(u',A')$, we deduce that
\begin{equation}\label{BoundFree}
F_\ve(u',A')+\frac12 \int_{\R^3\setminus\Omega} |\curl A'|^2\leq C\ve h_{\ex}F_\ve(u',A')+C\ve h_{\ex}^2F_\ve(u',A')^\frac12\leq C\ve h_{\ex}^2F_\ve(u',A')^\frac12.
\end{equation}
Inserting \eqref{first} and $h_{\ex}\leq \ve^{-\alpha}$ into \eqref{BoundFree}, we get 
$$
F_\ve(u',A')+\frac12 \int_{\R^3\setminus\Omega} |\curl A'|^2\leq C\ve^{1-3\alpha}\leq C\ve^{(1-2\alpha)\left(1+\frac12\right)}.
$$
By inserting this estimate into \eqref{BoundFree}, we obtain
$$
F_\ve(u',A')+\frac12 \int_{\R^3\setminus\Omega} |\curl A'|^2\leq C\ve^{(1-2\alpha)\left(1+\frac12+\frac14\right)}.
$$
Repeating this process a finite number of times, we are led to
$$
F_\ve(u',A')+\frac12 \int_{\R^3\setminus\Omega} |\curl A'|^2\leq C\ve^{(1-2\alpha)\left(1+\frac12+\frac14+\cdots+\frac1{2k}\right)}
$$
for some $k\in \N$. Since $\alpha<\frac14$ and $\sum_{i=0}^\infty \frac1{2^i}=2$, we deduce that 
$$
(1-2\alpha)\sum_{i=0}^k \frac1{2^i}>1
$$ 
for any $k$ sufficiently large. Hence
$$
F_\ve(u',A')+\frac12 \int_{\R^3\setminus\Omega} |\curl A'|^2\leq C\ve^{1+\de}
$$
for some $\de>0$.
\end{proof}
As a consequence, from Theorem \ref{theorem:uniqueness}, we obtain the uniqueness of the Meissner-type solution of Theorem \ref{theorem:meissner} for $\alpha<\frac14$.

\section*{Acknowledgments}
I am very grateful to my former Ph.D. advisors Etienne Sandier and Sylvia Serfaty for suggesting the problem and for useful comments.
Most of this work was done while I was a Ph.D. student at the Jacques-Louis Lions Laboratory of the Pierre and Marie Curie University, supported by a public grant overseen by the French National Research Agency (ANR) as part of the ``Investissements d'Avenir'' program (reference: ANR-10-LABX-0098, LabEx SMP). Part of this work was supported by the German Science Foundation DFG in the context of the Emmy Noether junior research group BE 5922/1-1. 

\bibliography{referencesFirstCriticalField}
\end{document}